\documentclass[a4paper]{amsart}

\usepackage[utf8]{inputenc} 
\usepackage[british]{babel} 
\usepackage{culmus} 
\usepackage{dsfont}

\usepackage{hyperref}

\begin{hyphenrules}{british}
\end{hyphenrules}

\usepackage{graphicx} 
\usepackage{booktabs} 
\usepackage{array} 
\usepackage{paralist} 
\usepackage{verbatim} 
\usepackage{subfig} 
\usepackage{scalerel,stackengine} 
\usepackage{bmpsize}
\usepackage{amssymb}
\usepackage{mathtools}
\usepackage{amsmath}	
\usepackage{amssymb}
\usepackage{amsthm}
\usepackage[mathcal]{eucal}
\usepackage{graphicx}
\usepackage{faktor} 
\usepackage{mathrsfs} 
\usepackage[all, cmtip]{xy} 
\usepackage{systeme}
\usepackage[titletoc]{appendix}
\usepackage{leftindex}
\usepackage{enumitem}
\usepackage{appendix}

\newcommand{\subscript}[2]{$#1 _ #2$}

\theoremstyle{plain}
\newtheorem{theorem}[subsection]{Theorem}
\newtheorem{lemma}[subsection]{Lemma}
\newtheorem{proposition}[subsection]{Proposition}
\newtheorem{corollary}[subsection]{Corollary}

\theoremstyle{definition}
\newtheorem{definition}[subsection]{Definition}

\theoremstyle{remark}

\newtheorem{examples}[subsection]{Examples}
\newtheorem{remark}[subsection]{Remark}

\stackMath
\newcommand\reallywidehat[1]{%
\savestack{\tmpbox}{\stretchto{%
  \scaleto{%
    \scalerel*[\widthof{\ensuremath{#1}}]{\kern-.6pt\bigwedge\kern-.6pt}%
    {\rule[-\textheight/2]{1ex}{\textheight}}
  }{\textheight}%
}{0.5ex}}%
\stackon[1pt]{#1}{\tmpbox}%
}
\newcommand{\pullback}[4]{{#1 \leftindex_{#2}{\times}_{#3} #4}}
\newcommand{\rel}[1]{\mathcal{R}_{\mathbb{#1}}} 
\newcommand{\R}[1]{R_{\mathbb{E}}}
\renewcommand{\P}[1]{P_{\mathbb{E}}}
\newcommand{\Ch}[1]{\mathfrak{Ch}_{\mathbb{#1}}} 
\newcommand{\la}{\langle}
\newcommand{\ra}{\rangle}
\usepackage{color}

\newcommand{\E}{\mathbb{E}}

\DeclareSymbolFont{alephbet}{HE8}{frank}{m}{n}
\SetSymbolFont{alephbet}{bold}{HE8}{frank}{b}{n}
\DeclareMathSymbol{\samech}{\mathord}{alephbet}{"F1}
\DeclareMathSymbol{\mem}{\mathord}{alephbet}{"EE}

\usepackage{chngcntr}
\counterwithin*{equation}{section}

\def\pullback{
 \ar@{-}[]+R+<4pt,-3pt>;[]+RD+<4pt,-6pt>%
 \ar@{-}[]+D+<3pt,-6pt>;[]+RD+<4pt,-6pt>}

\begin{document}

\title[The direction functor for Schreier extensions]{The direction functor for Schreier extensions of monoids}

\author[S.~Ambra]{Stefano Ambra}
\address[Stefano Ambra]{Dipartimento di Matematica ``Federigo Enriques'', Universit\`{a} degli Studi di Milano, Via Saldini 50, 20133 Milano, Italy}
\email{stefano.ambra@unimi.it}

\author[A.~Montoli]{Andrea Montoli}
\address[Andrea Montoli]{Dipartimento di Matematica ``Federigo Enriques'', Universit\`{a} degli Studi di Milano, Via Saldini 50, 20133 Milano, Italy}
\email{andrea.montoli@unimi.it}

\author[D.~Rodelo]{Diana Rodelo}
\address[Diana Rodelo]{Department of Mathematics, University of the Algarve, 8005-139 Faro, Portugal and Center for Research and Development in Mathematics and Applications (CIDMA), Department of Mathematics, University of Aveiro, 3810-193 Aveiro, Portugal} \email{drodelo@ualg.pt}

\thanks{This work was supported by the Shota Rustaveli National Science Foundation of Georgia (SRNSFG), through grant FR-24-9660, ``Categorical methods for the study of cohomology theory of monoid-like structures: an approach through Schreier extensions''.
The first and second authors are members of the Gruppo Nazionale per le Strutture Algebriche, Geometriche e le loro Applicazioni (GNSAGA) dell'Istituto Nazionale di Alta Matematica ``Francesco Severi''.
The third author acknowledges financial support by CIDMA (https://ror.org/05pm2mw36)
under the Portuguese Foundation for Science and Technology
(FCT, https://ror.org/00snfqn58), Grants UID/04106/2025 (https://doi.org/10.54499/UID/04106/2025)
and UID/PRR/04106/2025 (https://doi.org/10.54499/UID/PRR/04106/2025)}

\begin{abstract}
We observe that the process of associating an action to any Schreier extension of monoids with commutative and cancellative kernel is functorial. We show that this functor is a generalisation of the direction functor, used to give a categorical description of non-abelian cohomology in terms of extensions. We further prove that our functor is a conservative, product preserving cofibration and from this we conclude that its fibres are endowed with a canonical symmetric monoidal structure. The commutative monoids obtained as connected components of these symmetric monoidal categories are isomorphic to Patchkoria's second cohomology monoids of a monoid with coefficients in semimodules.
\end{abstract}



\maketitle

\section{Introduction}
The cohomology groups $H^n(B, A)$ of an abelian group $B$ with coefficients in an abelian group $A$ have a classical description in terms of exact sequences. Indeed, the first cohomology group $H^1(B, A)$ is isomorphic to the group of isomorphism classes of extensions of $B$ by $A$ (namely short exact sequences starting with $A$ and ending with $B$) with the Baer sum. A similar description of higher cohomology groups by means of $n$-extensions was obtained by Yoneda \cite{Yoneda}. The same results hold in any abelian category.

The situation for (non-abelian) group cohomology is more diversified. In fact, the whole set $\mathrm{Ext}(G, A)$ of isomorphism classes of extensions of a non-necessarily abelian group $G$ by an abelian group $A$ does not admit a cohomological interpretation. However, every group extension of $G$ by an abelian group $A$ induces a $G$-module structure on $A$; the set $\mathrm{Ext}(G, A)$ can then be partitioned into subsets of the form $\mathrm{Ext}(G, A, \varphi),$ where $\varphi \colon G \to \mathrm{Aut}(A)$ is an action of $G$ on $A,$ and, as shown in \cite{EML}, each set $\mathrm{Ext}(G, A, \varphi)$ is isomorphic to the second Eilenberg-Mac Lane cohomology group $H^2(G, A, \varphi)$ of $G$ with coefficients in the $G$-module $(A, \varphi).$ Interpretations of higher cohomology groups in terms of suitable exact sequences have been obtained in \cite{Holt, Huebschmann}. Similar results hold for different non-abelian algebraic structures, like associative algebras, Lie algebras, Leibniz algebras and many others.

The process of associating to a group extension with abelian kernel its induced action is functorial. The properties of this functor have been studied in detail in \cite{bourn-direction}, where this functor is seen as a particular instance of a very general categorical construction, called the \emph{direction functor}. The direction functor has as domain the category of objects with global support (meaning that the arrow to the terminal object is a regular epimorphism) which are endowed with an autonomous Mal'tsev operation, in a Barr-exact category $\mathcal{C};$ the codomain is the category of internal abelian groups in $\mathcal{C}.$ When applied to the slice category $\mathbf{Gp}/G,$ where $\mathbf{Gp}$ is the category of groups, the direction functor is precisely the functor sending every extension of $G$ with abelian kernel to the corresponding action. Indeed, the objects with global support in $\mathbf{Gp}/G$ are the surjective group homomorphisms $f \colon X \to G,$ and such an $f$ is endowed with a, necessarily autonomous, Mal'tsev operation if and only if its kernel congruence $\mathcal{E}q(f)$ centralises itself in the sense of \cite{pedicchio}, which is equivalent to saying that the kernel of $f$ is abelian. The double equivalence relation determined by the fact that $\mathcal{E}q(f)$ centralises itself, called the \emph{Chasles relation} in \cite{bourn-direction}, plays a key role in the construction of the direction functor: in fact, the image of $f$ under the direction functor is the split epimorphism with codomain $G$ and domain the coequalizer of the Chasles relation. Such a split epimorphism, having an abelian kernel, is an internal abelian group in $\mathbf{Gp}/G.$ The action corresponding to such split epimorphism, via the semidirect product construction, is precisely the action, induced by $f,$ of $G$ on the kernel of $f.$

It is shown in \cite{bourn-direction} that the direction functor has many good properties; in particular, it is a cofibration, it is conservative and it preserves finite products. Thanks to these properties, the abelian group structure on any object in its codomain can be lifted to a symmetric monoidal closed structure on its fibre. On the sets of connected components of such fibres there result then canonical abelian group structures. In the case $\mathcal{C} = \mathbf{Gp}/G,$ these abelian groups are precisely the cohomology groups $H^2.$ In \cite{Bourn-Rodelo} a similar description of higher cohomology groups using higher dimensional direction functors has been obtained for Barr-exact categories (actually Barr-exactness can be weakened, considering instead \emph{efficiently regular} categories, including more examples, like the category of topological groups; see \cite{Bourn-Rodelo}).

There have been different attempts to give a cohomological description of extensions of monoids, too. A crucial notion in this respect is the one of Schreier extension of monoids, introduced in \cite{Redei}. Schreier extensions of monoids retain several good properties of group extensions. It was shown in \cite{Hoff, MMS_Baer_sum, MMS_pf} that the set of isomorphism classes of Schreier extensions, inducing the same action, of a monoid $M$ by an abelian group $A$ (called \emph{special Schreier} extensions in \cite{schreier_book, BMMS_sem_forum}) has an abelian group structure which extends the usual Baer sum of group extensions, and such abelian groups are actually isomorphic to the cohomology groups $H^2(M, A, \varphi)$ of a monoid $M$ with coefficients in an $M$-module $(A, \varphi)$ of a cohomology theory which generalises straightforwardly the Eilenberg-Mac Lane cohomology of groups. Patchkoria introduced in \cite{P-I} a cohomology theory for monoids with coefficients in semimodules (which is a variation on yet another cohomology theory he introduced in \cite{Patchkoria77}), in which one gets commutative monoids of cohomology, instead of abelian groups. Moreover, he showed in \cite{P-II} that the second cohomology monoids $H^2(M, K, \varphi)$ of his theory, where $M$ is a monoid and $(K, \varphi)$ is a cancellative $M$-semimodule, describe Schreier extensions of $M$ by $K.$

In the present paper we show that, also in the case of monoids, the process of associating an action to a Schreier extension with commutative kernel is functorial and that this functor is, in a sense, a generalisation of the direction functor studied in \cite{bourn-direction}. The domain of our functor is the category $cc\text{-}SExt_M$ whose objects are the Schreier extensions of a monoid $M$ with commutative and cancellative kernel. To such an extension $\mathbb{E} : \xymatrix{K \ar@{>->}[r]^-k &X \ar@{->>}[r]^-f &M}$ we associate a Schreier reflexive relation (in the sense of \cite{schreier_book, BMMS_sem_forum}) $\mathcal{R}_{\mathbb{E}},$ which is automatically transitive (thanks to the results in \cite{schreier_book, BMMS_sem_forum}) but not symmetric, in general: the congruence generated by $\mathcal{R}_{\mathbb{E}}$ is precisely the kernel congruence $\mathcal{E}q(f)$ of $f,$ and $\mathcal{R}_{\mathbb{E}}$ and $\mathcal{E}q(f)$ coincide if and only if the kernel of $f$ is a group. We observe that the Schreier relation $\mathcal{R}_{\mathbb{E}}$ centralises itself; the corresponding double Schreier reflexive relation (whose existence and uniqueness are guaranteed by the results in \cite{BMMS_S-protomod}) is then a generalisation of the Chasles relation of \cite{bourn-direction}, and our functor associates with $\mathbb{E}$ a split epimorphism with codomain $M$ and domain the coequalizer of the generalised Chasles relation. This image is an internal commutative monoid in the category whose objects are Schreier extensions of a monoid $M.$ The result is then a functor
\[ d \colon cc\text{-}SExt_M \longrightarrow \mathbf{CMon}(cc\text{-}SExt_M) \]
whose codomain is the category of internal commutative monoids in $cc\text{-}SExt_M.$ We show that this functor still satisfies the key properties of the direction functor from \cite{bourn-direction}. This allows us to equip the fibres of the functor with symmetric monoidal structures in a way that the sets of connected components of such fibres turn out to be endowed with canonical commutative monoid structures. The so obtained monoids are precisely Patchkoria's second cohomology monoids built in \cite{P-I}.
The functor $d$ can actually be extended to a functor
\[ D \colon smod\text{-}SExt_M \longrightarrow \mathbf{CMon}(smod\text{-}SExt_M), \]
where $smod\text{-}SExt_M$ is the category whose objects are the Schreier extensions of $M$ inducing an $M$-semimodule structure ($cc\text{-}SExt_M$ is a full subcategory of $smod\text{-}SExt_M$). This broader functor shares the same properties of $d,$ hence the connected components of its fibres can also be equipped with canonical commutative monoid structures, that correspond to the second cohomology monoids of the cohomology theory considered by Patchkoria in \cite{Patchkoria77}.

\section{Definition and first properties of Schreier extensions of monoids}
\label{section:Definition and first properties of Schreier extensions of monoids}
In this paper we shall be interested in monoid extensions, namely sequences of monoids and monoid homomorphisms
\begin{equation}
\label{eqn:ses}
\xymatrix{{\mathbb{E}:K} \ar@{>->}[r]^-k &X \ar@{->>}[r]^-f &M}
\end{equation}
in which $(K,k)$ is a kernel of $f$ and $f$ is a regular epimorphism (i.e. surjective monoid homomorphism). Following the choice made in \cite{ML_Homology} for group extensions, we shall adopt an additive notation on $K$ and $X$ and a multiplicative notation on $M.$ It should be clear, nonetheless, that while we will eventually require $K$ to be commutative, no such assumption is ever made on $X,$ in general. We write $Ker(f)=k(K)=\{x\in X: f(x)=1\}.$ The category of monoids and monoid homomorphisms shall be denoted by $\mathbf{Mon}.$ \\

Given a monoid extension~\eqref{eqn:ses}, denote by $B_m,$ for $m\in M,$ the subset of $f^{-1}(m)$ defined by
\[
B_m=\Big\{u\in f^{-1}(m):\forall x\in f^{-1}(m), \ \exists! \, a\in K \ \text{such that} \ x=k(a)+u\Big\}.
\]
Observe that $B_m$ may very well be empty for some $m,$ but since $k$ is a monomorphism, one always has at least $0\in B_1.$ The elements of $B_m$ shall be called the \emph{representatives} of $m,$ and we shall denote by $B(\mathbb{E})=\bigcup_{m\in M}B_m$ the set of all representatives. It is a subset of $X,$ by construction.

In the next lemma, we collect some useful properties which immediately stem out of these definitions.
\begin{lemma}
\label{lemma:rep_invertible}
In the above notation:
\begin{enumerate}
	\item[\em(1)] For every $a,a^\prime\in K,$ $m\in M$ and $u\in B_m,$ the equality $k(a)+u=k(a^\prime)+u$ holds if and only if $a=a^\prime.$
	\item[\em(2)] If $u,v\in B_m,$ one has $u=k(a)+v$ for a unique $a\in U(K)$ (where $U(K)\subseteq K$ denotes the subgroup of invertible elements).
	\item[\em(3)] Conversely, if $u\in B_m$ and $a\in U(K),$ then $k(a)+u\in B_m.$
\end{enumerate}
\end{lemma}

\begin{definition}[\cite{Redei}]
\label{def:schreier}
We shall say that~\eqref{eqn:ses} is a \emph{Schreier extension} if $B_m\neq \emptyset$  for every $m\in M.$
\end{definition}

If~\eqref{eqn:ses} is a Schreier extension, then for every fixed $m\in M$ and $u\in B_m$ the equation $x=k\big(q_{m,u}(x)\big)+u$ defines a map of \emph{sets} (denoted with a dashed arrow in this work to emphasise that it is not a morphism in $\mathbf{Mon}$)
\[
\xymatrix{{q_{m,u}\colon f^{-1}(m)} \ar@{-->}[r] &{K, \ x\mapsto q_{m,u}(x).}}
\]
\begin{corollary}
\label{cor:group_B}
If~\eqref{eqn:ses} is a monoid extension, for every $m\in M$ and every $u\in B_m$ there is a bijection
\[
\xymatrix{{\vartheta=\vartheta_{m,u}\colon U(K)} \ar@{-->}[r]^-{\sim} &{B_m, \ a\mapsto k(a)+u.}}
\]
Thus, for any $u\in B_m,$ there results a group structure on $B_m$ having $u$ as neutral element, which is abelian if $U(K)$ is abelian. In particular, if~\eqref{eqn:ses} is a Schreier extension, for all $m,m^\prime\in M$ there is a bijection $B_m\cong B_{m^\prime},$ which is a group isomorphism with respect to the group structures induced by the $\vartheta$'s.
\end{corollary}
\begin{proof}
By Lemma~\ref{lemma:rep_invertible}, $\vartheta_{m,u}$ is injective (point $(1)$) and surjective (point $(2)$). The desired group structure is then obtained by transferring on $B_m$ the group structure of $U(K)$ via $\vartheta_{m,u},$ i.e. by defining the sum $v+_uw=\vartheta_{m,u}\big(\vartheta_{m,u}^{-1}(v)+\vartheta_{m,u}^{-1}(w)\big)$ for all $v,w\in B_m.$
\end{proof}

Some other immediate consequences of Definition~\ref{def:schreier} are collected in the following lemma.
\begin{lemma}[cf. {\cite[Proposition 2.1.5]{schreier_book}}] For a Schreier extension~\eqref{eqn:ses}:
\label{prop:2.1.5}
\begin{enumerate}
	\item[\em{(1)}] $q_{1,0}k=id_K;$
	\item[\em{(2)}] for every $u\in B(\mathbb{E}),$ $q_{f(u),u}(u)=0;$
	\item[\em{(3)}] for every $a\in K$ and $u\in B(\mathbb{E}),$ $u+k(a)=k\Big(q_{f(u),u}\big(u+k(a)\big)\Big)+u;$
	\item[\em{(4)}] for every $x,x^\prime\in X,$ $u\in B_{f(x)}$ and $u^\prime\in B_{f(x^\prime)},$
\[
q_{f(x),u}(x)+q_{f(x^\prime),u^\prime}(x^\prime)=q_{f(x^\prime),u^\prime}\Big(k\big(q_{f(x),u}(x)\big)+x^\prime\Big).
\]
\end{enumerate}
\end{lemma}
\begin{proof}
\begin{enumerate}
	\item Given $a\in K,$ the element $q_{1,0}\big(k(a)\big)$ is uniquely determined by $k(a)=k\Big(q_{1,0}\big(k(a)\big)\Big),$ and $k$ is a monomorphism.
	\item One has $u\in B_m$ for some $m\in M,$ so that $f(u)=m$; thus $u$ is a representative of $f(u).$ Also, $u=k(0)+u$ and $q_{f(u),u}(u)$ is uniquely determined by $u=k\big(q_{f(u),u}(u)\big)+u.$ The equality follows from the definition of a Schreier extension.
	\item Immediate by observing that $u$ is a representative of $f\big(u+k(a)\big)=f(u).$
	\item The element $q_{f(x^\prime),u^\prime}\Big(k\big(q_{f(x),u}(x)\big)+x^\prime\Big)\in K$ is uniquely determined by the equation $k\big(q_{f(x),u}(x)\big)+x^\prime=k\Big(q_{f(x^\prime),u^\prime}\Big(k\big(q_{f(x),u}(x)\big)+x^\prime\Big)\Big)+u^\prime.$ Then
\begin{equation*}
\begin{split}
k\big(q_{f(x),u}(x)+q_{f(x^\prime),u^\prime}(x^\prime)\big)+u^\prime&=k\big(q_{f(x),u}(x)\big)+k\big(q_{f(x^\prime),u^\prime}(x^\prime)\big)+u^\prime\\
&=k\big(q_{f(x),u}(x)\big)+x^\prime,
\end{split}
\end{equation*}
and the desired equality holds from the definition of a Schreier extension.
\end{enumerate}
\end{proof}

From now on, in order to simplify the notation, we shall write $u_m$ to denote some representative of $m\in M$ and $q_{m,u}=q$ for all $m\in M$ and $u\in B_m.$ Thus, for example, given a Schreier extension~\eqref{eqn:ses} we have for every $x\in X$ (and $u_{f(x)}\in B_{f(x)}$) a unique $q(x)\in K$ such that
\begin{equation}
\label{eqn:x=kq(x)+u}
	x = kq(x)+u_{f(x)},
\end{equation}
and the equality in point $(4)$ of the previous lemma shall be expressed simply as $q(x)+q(x^\prime)=q\big(kq(x)+x^\prime\big).$

Since by Lemma~\ref{lemma:rep_invertible} any two representatives of the same $m$ differ by an invertible element, we have a characterisation of the Schreier extensions~\eqref{eqn:ses} in which the kernel is a \emph{group}.
\begin{proposition}
\label{prop:K_group}
Given a Schreier extension~$\eqref{eqn:ses},$ the following are equivalent:
\begin{enumerate}
	\item[\em{(1)}] $B(\mathbb{E})=X,$ i.e. $B_m=f^{-1}(m)$ for every $m\in M;$
	\item[\em{(2)}] $B_1=Ker(f);$
	\item[\em{(3)}] $K$ is a group.
\end{enumerate}
\end{proposition}
\begin{proof}
It is clear that $(1)$ implies $(2),$ and $(3)$ follows from $(2)$ by setting $u=k(a)$ and $v=0$ in Lemma~\ref{lemma:rep_invertible}$(2),$ for any $a\in K.$ Now, to prove that $(3)$ implies $(1),$ it is enough to show that if $a\in K$ and $u_m$ is a representative, then $k(a)+u_m$ is also a representative: but if $x\in f^{-1}(m),$ then $x=kq(x)+u_m$ by~\eqref{eqn:x=kq(x)+u}, so that $x=k(a^{\prime\prime})+(k(a)+u_m)$ for $a^{\prime\prime}\in K$ uniquely determined as $a^{\prime\prime}=q(x) -a.$
\end{proof}

Another relevant consequence of Definition~\ref{def:schreier} is the following one which proves that Schreier extensions are, indeed, short exact sequences.
\begin{proposition}[{Cf. \cite[Section 4]{P-II}}]
\label{prop:f_normal}
If~\eqref{eqn:ses} is a Schreier extension, then $f$ is a cokernel of $k,$ so that~\eqref{eqn:ses} is a short exact sequence in $\mathbf{Mon}.$
\end{proposition}
\begin{proof}
Let $g:X\rightarrow Z$ be a monoid homomorphism such that $gk=0.$ Then $g$ is constant on every fibre $f^{-1}(m)$ of $f.$ Indeed, since~\eqref{eqn:ses} is a Schreier extension we have $B_m\neq \emptyset$; let $u_m\in B_m.$ If $x, x'\in f^{-1}(m),$ then $x=kq(x)+u_m$ and $x'=kq(x')+u_m$ by~\eqref{eqn:x=kq(x)+u}. It follows that $g(x)=g(u_m)=g(x').$ Note that, if $u_m,v_m$ are two representatives of $m,$ then $g(u_m)=g(v_m),$ as $u_m,v_m\in f^{-1}(m).$ Define $\varphi\colon M\rightarrow Z$ by setting $\varphi(m)=g(u_m),$ where $u_m$ is any representative of $m.$ We proved above that $\varphi$ is well defined and, as a function, it clearly satisfies $\varphi f=g.$ To see that it is also a monoid homomorphism, let $m,m^\prime\in M$: then $u_m+u_{m^\prime}=kq(u_m+u_{m'})+u_{m\cdot m^\prime},$ again by~\eqref{eqn:x=kq(x)+u}, and by using the fact that $gk=0$ we have
\begin{equation*}
\begin{split}
\varphi(m\cdot m^\prime)&=g(u_{m\cdot m^\prime})=g(kq(u_m+u_{m'})+u_{m\cdot m^\prime})\\
&=g(u_m+u_{m^\prime})=g(u_m)\cdot g(u_{m^\prime})=\varphi(m)\cdot\varphi(m^\prime).
\end{split}
\end{equation*}
The uniqueness of $\varphi$ is immediate, since $f$ is epimorphic.
\end{proof}

It will be shown in Example~\ref{ex:schreier_extensions}.6 that the converse of this result does not hold in general.

The notion of Schreier extension comes with a coherent notion of morphism between Schreier extensions:

\begin{definition}
\label{dfn:morphism of Schreier exts}
A \emph{morphism of Schreier extensions} $\mathbb{E}\rightarrow\mathbb{E}^\prime$ is a morphism of monoid extensions
\begin{equation}
\begin{aligned}
\label{eqn:morphism_of_ses}
\xymatrix{
{\mathbb{E}:K} \ar@{}[rd]|{\mathrm{(a)}} \ar@<1.6ex>[d]_-{\alpha_1} \ar@{>->}[r]^-k &X \ar@{}[rd]|{\mathrm{(b)}} \ar[d]_-{\alpha_2} \ar@{->>}[r]^-f &M \ar[d]^-{\alpha_3} \\
{\mathbb{E^\prime}:K^\prime} \ar@{>->}[r]_-{k^\prime} &{X^\prime} \ar@{->>}[r]_-{f^\prime} &{M^\prime},
}
\end{aligned}
\end{equation}
meaning that $\alpha_1,\alpha_2,\alpha_3$ are monoids homomorphisms such that the squares $\mathrm{(a)}$ and $\mathrm{(b)}$ commute, satisfying the additional condition $\alpha_2(B(\mathbb{E}))\subseteq B(\mathbb{E}^\prime).$
\end{definition}

By the commutativity of $\mathrm{(b)},$ the restriction of $\alpha_2$ to $B_m$ gives a map $\alpha_2\colon B_m\rightarrow B^\prime_{\alpha_3(m)}$ (the latter denoting the set of representatives of $\alpha_3(m)$ for $\E'$), which is a group homomorphism with respect to the group structures of Corollary~\ref{cor:group_B}, because for every $u\in B_m$ the square
\[
\xymatrixcolsep{4pc}
\xymatrix{
{U(K)} \ar[d]_-{\alpha_1} \ar[r]_-{\sim}^-{\vartheta_{m,u}} &{B_m} \ar[d]^-{\alpha_2} \\
{U(K^\prime)} \ar[r]^-{\sim}_-{\vartheta_{\alpha_3(m),\alpha_2(u)}} &{B^\prime_{\alpha_3(m)}}
}
\]
commutes (using the commutativity of $\mathrm{(a)}$ in~\eqref{eqn:morphism_of_ses}).

It is a consequence of Lemma~\ref{lemma:rep_invertible} that, in order to prove that a morphism of monoid extensions~\eqref{eqn:morphism_of_ses} between Schreier extensions is a morphism of Schreier extensions, it is enough to show that for all $m\in M$ \emph{some} representative $u_m$ of $m$ in $\mathbb{E}$ is mapped by $\alpha_2$ to a representative in $\mathbb{E^\prime}$:
\begin{proposition}
\label{prop:rep_pres}
Consider a morphism~\eqref{eqn:morphism_of_ses} of (not necessarily Schreier) monoid extensions
and suppose that $\alpha_2(u_m)\in B^\prime_{\alpha_3(m)}$ for some $u_m\in B_m$ (where $m\in M$). Then $\alpha_2(v_m)\in B^\prime_{\alpha_3(m)}$ for all $v_m\in B_m.$
\end{proposition}
\begin{proof}
If $u_m,v_m\in B_m,$ by Lemma~\ref{lemma:rep_invertible}$(2)$ we have $v_m=k(a)+u_m$ for a unique $a\in U(K),$ and $f^\prime\big(\alpha_2(v_m)\big)=\alpha_3\big(f(v_m)\big)=\alpha_3(m).$ Moreover, $\alpha_1(a)$ is invertible in $K^\prime$ because $a$ is invertible in $K$ and $\alpha_1$ is a monoid homomorphism. We have $\alpha_2(v_m)=\alpha_2\big(k(a)\big)+\alpha_2(u_m)=k^\prime\big(\alpha_1(a)\big)+\alpha_2(u_m),$ which proves that $\alpha_2(v_m)\in B^\prime_{\alpha_3(m)}$ by Lemma~\ref{lemma:rep_invertible}$(3).$
\end{proof}

One can prove that the Short Five Lemma, which is \emph{not} valid in general for monoids (as the category $\mathbf{Mon}$ is not protomodular \cite{bourn-proto}), does hold for Schreier extensions.

\begin{proposition}[{\cite[Proposition 4.5]{P-II}}]
\label{prop:short_five_lemma_schreier}
Consider a morphism~\eqref{eqn:morphism_of_ses} of Schreier extensions. Then:
\begin{enumerate}
	\item[\em{(1)}] If $\alpha_1$ and $\alpha_3$ are monomorphisms (i.e. injective monoid homomorphisms), $\alpha_2$ is a monomorphism;
	\item[\em{(2)}] If $\alpha_1$ and $\alpha_3$ are regular epimorphisms (i.e. surjective monoid homomorphisms), $\alpha_2$ is a regular epimorphism.
\end{enumerate}
Thus, in particular, if $\alpha_1$ and $\alpha_3$ are isomorphisms, so is $\alpha_2.$
\end{proposition}

We conclude this section with some concrete examples of Schreier and non-Schreier monoid extensions.
\begin{examples}
\label{ex:schreier_extensions}
\begin{enumerate}[label=\arabic*.]
	\item Every short exact sequence of groups
\[ \xymatrix{K \ar@{>->}[r]^-k &X \ar@{->>}[r]^-f &G} \]
is a Schreier extension of monoids in which all $x\in X$ are representatives, by Proposition~\ref{prop:K_group}. For any $u\in X$ and any $x\in f^{-1}(f(u)),$ one has indeed $x=k(a)+u$ for a unique $a\in K$ given by $k(a)=x-u.$
	\item For every monoid $M,$ the trivial extension
\[
\xymatrix{{M} \ar@{=}[r] &{M} \ar@{>>}[r] &0}
\]
is a Schreier extension whose representatives are the invertible elements of $M.$ Similarly, the extension
\[
\xymatrix{{0} \ar@{>->}[r] &M \ar@{=}[r] &{M}}
\]
is a Schreier extension, with $B_m=\{m\}$ for all $m\in M.$
	\item For all monoids $K$ and $M,$ the monoid extension
\[
\xymatrixcolsep{2.5pc}
\xymatrix{K \ar@{>->}[r]^-{\la id_K,0\ra} &{K\times M} \ar@{>>}[r]^-{\pi_M} &M}
\]
(where $\pi_M(a,m)=m$) is a Schreier extension, having as representatives the couples $(a,m)$ with $m\in M$ and $a\in U(K).$
	\item (Cf.\cite[Example 4.2]{P-II}) Let $m>1$ be a natural number and consider the monoid extension
\[ \xymatrix{{(\mathbb{N},+,0)} \ar@{>->}[r]^-{\cdot m} &{(\mathbb{N},+,0)} \ar@{->>}[r]^-f &{C_m(t),}} \]
where $C_m(t)$ is the multiplicative cyclic group of order $m$ with generator $t,$ $\cdot m$ is the ordinary multiplication by $m$ and $f(1)=t.$ It is a Schreier extension, with representatives $B_{1}=\{0\},B_{t}=\{1\},\dots,B_{t^{m-1}}=\{m-1\}.$
	\item The monoid extension
\[ \xymatrix{0 \ar@{>->}[r] &{\mathbb{N}\times\mathbb{N}} \ar@{>>}[r]^-{+} &{\mathbb{N}}} \]
is \emph{not} a Schreier extension by Proposition~\ref{prop:f_normal}, because the regular epimorphism $+$ given by the sum of natural numbers is not a normal epimorphism (otherwise it would be a cokernel of its kernel $0,$ and thus an isomorphism, but $+$ is not injective).
	\item Let $(M_2=\{0,1\},\cdot,1)$ be the commutative monoid with unit $1$ and $0\cdot0=0,$ and consider the monoid extension
\[ \xymatrix{ {(\mathbb{N}\setminus\{0\},\cdot,1)}\ar@{>->}[r]^-i &{(\mathbb{N},\cdot,1)} \ar@{->>}[r]^-f &{(M_2,\cdot,1)}} \]
where $i$ is the inclusion morphism, $f(n)=1$ if $n\neq 0$ and $f(0)=0.$ This extension is \emph{not} a Schreier extension, because there are infinitely many $a\in \mathbb{N}\setminus\{0\}$ such that $0=a\cdot 0.$ Even so, $f$ is a normal epimorphism in $\mathbf{Mon}$: indeed, if $g\colon(\mathbb{N},\cdot,1)\rightarrow (M,\cdot,1)$ is any monoid homomorphism such that $g(n)=1$ for all $n\in\mathbb{N}\setminus\{0\},$ $g$ factors through $f$ as $g=\overline{g}f$ (uniquely so, as $f$ is epimorphic) with $\overline{g}\colon M_2\rightarrow M$ given by $\overline{g}(1)=1$ and $\overline{g}(0)=g(0).$ Thus the converse of Proposition~\ref{prop:f_normal} does not hold.
\end{enumerate}
\end{examples}

\section{Action induced by a Schreier extension}
\label{sec:action}
It is well known that if $\xymatrix{K \ar@{>->}[r]^-k &X \ar@{->>}[r]^-f &G}$ is a short exact sequence of groups in which $K$ is abelian, then there results a (left) action of $G$ on $K$ given by
\begin{equation}
\label{eqn:gr_action}
G\times K\dashrightarrow K, \ (g,a)\mapsto a^\prime \text{ with } k(a^\prime)=s(g)+k(a)-s(g),
\end{equation}
where $s\colon G\dashrightarrow X$ is any set-theoretic section of $f$ (i.e. $s$ is any map of sets such that $fs=id_G$). We want to show that something similar happens for Schreier extensions of monoids. 

Most facts in this section are already known, and the statements can be found, for example, in \cite{P-II}, but since these properties will have a central importance in our main construction, we find it appropriate to provide the reader with explicit proofs.

Recall that a (left) action of a monoid $(M,\cdot,1)$ on a monoid $(A,+,0)$ is a monoid homomorphism $(M,\cdot,1)\rightarrow(\mathrm{End}(A),\circ,id_A),$ where $\mathrm{End}(A)$ denotes the set of monoid endomorphisms of $A;$ this is equivalent to having a map of sets $M\times A\dashrightarrow A,$ $(m,a)\mapsto m\ast a,$ satisfying the four axioms:
\begin{enumerate}[label=(\subscript{A}{{\arabic*}})]
\item $1\ast a=a$ for all $a\in A;$
\item $m\ast 0=0$ for all $m\in M;$
\item $m\ast (a+a^\prime)=m\ast a+m\ast a^\prime$ for all $m\in M$ and $a,a^\prime\in A;$
\item $(m^\prime\cdot m)\ast a=m^\prime\ast(m\ast a)$ for all $m,m^\prime\in M$ and $a\in A.$
\end{enumerate}

If $\xymatrix{K \ar@{>->}[r]^-k &X \ar@{->>}[r]^-f &M}$ is a Schreier extension of monoids with $K$ commutative, we claim that the map
\begin{equation}
\label{eqn:action}
M\times K\dashrightarrow K, \ (m,a)\mapsto m\ast a \ \text{ with } u_m+k(a)=k(m\ast a)+u_m
\end{equation}
is well defined, meaning that it does not depend on the choice of the representative $u_m$ of $m.$ Indeed, let $u,v\in B_m,$ $a\in K,$ and let $a_u,a_v\in K$ be such that $u+k(a)=k(a_u)+u$ and $v+k(a)=k(a_v)+v.$ As both $u$ and $v$ are representatives of $m,$ by Lemma~\ref{lemma:rep_invertible}$(2)$ there exists a unique $a^\prime\in U(K)$ such that $u=k(a^\prime)+v.$ We have
\begin{equation*}
\begin{split}
u+k(a)=k(a^\prime)+v+k(a)=k(a^\prime)+k(a_v)+v=k(a^\prime+a_v)+v, \\
\\
u+k(a)=k(a_u)+u=k(a_u)+k(a^\prime)+v=k(a_u+a^\prime)+v,
\end{split}
\end{equation*}
and by Lemma~\ref{lemma:rep_invertible}$(1),$ the commutativity of $K$ and the invertibility of $a^\prime,$ it follows that $a_u=a_v.$

Observe that~\eqref{eqn:action} can be expressed by
\begin{equation}
\label{eqn:action with q}
	m\ast a=q\big(u_m+k(a)\big),
\end{equation}
in the notation of the previous section.

The key point here is that~\eqref{eqn:action} \emph{almost} defines an action of $M$ on $K,$ as shown next.

\begin{proposition}
\label{prop:monoid action}
If $\xymatrix{{\mathbb{E}: K} \ar@{>->}[r]^-k &X \ar@{->>}[r]^-f &M}$ is a Schreier extension in which $K$ is commutative, the equation~\eqref{eqn:action} always satisfies the axioms $(A_1),$ $(A_2)$ and $(A_3).$ If, moreover, $K$ is cancellative, then $(A_4)$ is satisfied too, and in this case~\eqref{eqn:action} defines a monoid action of $M$ on $K.$
\end{proposition}

\begin{proof}
\begin{enumerate}[label=(\subscript{A}{{\arabic*}})]
\item By definition, for $a\in K,$ $1\ast a$ is the only element satisfying $u_1+k(a)=k(1\ast a)+u_1;$ then by setting $u_1=0$ we get $1\ast a=a.$
\item For $m\in M,$ it follows from Lemma~\ref{lemma:rep_invertible}$(1)$ that $u_m=u_m+k(0)=k(m\ast0)+u_m$ entails $m\ast0=0.$
\item Given $m\in M$ and $a,a^\prime\in K,$ let $m\ast a=b$ and $m\ast a^\prime=b^\prime,$ so that $u_m+k(a)=k(b)+u_m$ and $u_m+k(a^\prime)=k(b^\prime)+u_m.$ Then
\begin{equation*}
\begin{split}
u_m+k(a+a^\prime)&=u_m+k(a)+k(a^\prime)\\
&=k(b)+u_m+k(a^\prime)\\
&=k(b)+k(b^\prime)+u_m\\
&=k(b+b^\prime)+u_m,
\end{split}
\end{equation*}
whence $m*(a+a^\ast)=b+b^\prime=m\ast a+m\ast a^\prime.$
\item Now suppose that $K$ is also cancellative and define $m\ast a=b$ (so that $u_m+k(a)=k(b)+u_m$), $m^\prime\ast b=b^\prime$ (so that $u_{m^\prime}+k(b)=k(b^\prime)+u_{m^\prime}$), and $(m^\prime\cdot m)\ast a=c$ (so that $u_{m^\prime\cdot m}+k(a)=k(c)+u_{m^\prime\cdot m}$). We want to prove that $c=b^\prime.$ Write $u_{m^\prime}+u_m=k(l)+u_{m^\prime\cdot m}$; then
\begin{equation*}
\begin{split}
k(l+b^\prime)+u_{m^\prime\cdot m}&=k(b^\prime)+k(l)+u_{m^\prime\cdot m}\\
&=k(b^\prime)+u_{m^\prime}+u_m\\
&=u_{m^\prime}+k(b)+u_m\\
&=u_{m^\prime}+u_m+k(a)\\
&=k(l)+u_{m^\prime\cdot m}+k(a)\\
&=k(l)+k(c)+u_{m^\prime\cdot m}\\
&=k(l+c)+u_{m^\prime\cdot m},
\end{split}
\end{equation*}
and the result follows from Lemma~\ref{lemma:rep_invertible}(1) and the cancellativity of $K.$
\end{enumerate}
\end{proof}

We stress the fact that the cancellativity of $K$ is only a sufficient condition for~\eqref{eqn:action} to be a monoid action: indeed, observe that if $k(K)$ is central in $X,$ then~\eqref{eqn:action} is just the trivial assignment $m\ast a=a$ for all $m\in M$ and $a\in K,$ which is always an action (the trivial one) even if $K$ is not cancellative. This is the case, for instance, of Example~\ref{ex:schreier_extensions}.2 (first example) when $M$ is commutative but not cancellative. Similarly, the Schreier extension of Example~\ref{ex:schreier_extensions}.3 always induces the trivial action when $K$ is commutative (even if $M$ is not commutative).

Given a Schreier extension $\E$ as in Proposition~\ref{prop:monoid action} with commutative and cancellative $K,$ we denote the monoid homomorphism corresponding to the above action by
\begin{equation}
\label{eqn:Schreier ext action eta}
	\eta\colon M \to \mathrm{End}(K), \ m \mapsto \eta(m)\colon K \to K,
\end{equation}
where $\eta(m)(a)=m\ast a,$ as in~\eqref{eqn:action}.

Observe, moreover, that when $\xymatrix{K \ar@{>->}[r]^-k &X \ar@{->>}[r]^-f &G}$ is a short exact sequence of groups with an abelian kernel (which is always a Schreier extension of monoids, with $K$ cancellative), the induced action~\eqref{eqn:action} does coincide with the group action~\eqref{eqn:gr_action}.

Since the request for $K$ to be commutative and cancellative will turn out repeatedly in the next sections, we set for the sake of brevity the following definition.
\begin{definition}
A \emph{cc-Schreier extension} is a Schreier extension $\xymatrixcolsep{1.6pc}\xymatrix{K \ar@{>->}[r]^-k &X \ar@{->>}[r]^-f &M}$ in which the kernel $K$ is commutative and cancellative.
\end{definition}

We conclude this section by pointing out a property of the map~\eqref{eqn:action} which, despite its simplicity, will be very useful in carrying out explicit computations.

\begin{lemma}
\label{lemma:patrick}
If $\xymatrix{K \ar@{>->}[r]^-k &X \ar@{->>}[r]^-f &M}$ is a Schreier extension with $K$ commutative, then
\begin{equation}
\label{eqn:Patrick}
x+k(a)=k\big(f(x)\ast a\big)+x
\end{equation}
for all $x\in X$ and $a\in K.$
\end{lemma}

\begin{proof}
We have $x=kq(x)+u_{f(x)},$ from~\eqref{eqn:x=kq(x)+u}. Then $x+k(a)=kq(x)+u_{f(x)}+k(a)\stackrel{\eqref{eqn:action}}{=}kq(x)+k\big(f(x)\ast a\big)+u_{f(x)}=k\big(f(x)\ast a\big)+kq(x)+u_{f(x)}=k\big(f(x)\ast a\big)+x,$ using the commutativity of $K.$
\end{proof}

\section{Interlude on Schreier points and $S$-reflexive relations}
In this section we collect some of the main results concerning the so-called Schreier \emph{points}, which shall be needed in the course of the construction of our direction functor.

The following notion was introduced in \cite{mms}, Definition $2.6$:
\begin{definition}
\label{def:schreier_point}
A \emph{Schreier split extension} of monoids, or \emph{Schreier point}, is a split extension
\begin{equation}
\label{eqn:schreier_point}
\xymatrix{{K} \ar@{>->}[r]^-k &B \ar@<.5ex>@{>>}[r]^-f &M \ar@<.5ex>[l]^-s}
\end{equation}
in $\mathbf{Mon}$ (where the monoid homomorphism $s$ is a fixed section of $f$ and $(K,k)$ is a kernel of $f$) for which there exists a unique map of sets $q:B\dashrightarrow K$ (called the \emph{Schreier retraction}) satisfying
\begin{equation}
\label{eqn:equality for Schreier point}
	b=kq(b)+sf(b),
\end{equation}
for every $b\in B.$
\end{definition}

An extensive study of Schreier split extensions of monoids is carried out in \cite{schreier_book}.

\begin{remark}
\label{rmk:Shreier exts vs Schreier points}
Observe that any Schreier split extension of monoids~\eqref{eqn:schreier_point} determines a Schreier extension $\xymatrix{{K} \ar@{>->}[r]^-k &B \ar@{>>}[r]^-f &M}$ in the sense of Definition~\ref{def:schreier}, with $s(m)\in B_m$ for all $m\in M.$ Beware that $B_m$ need not be reduced to the sole element $s(m),$ as the case of split extensions of groups shows. Indeed, by Proposition~\ref{prop:K_group} we have $B_1=f^{-1}(1)=k(K)\neq \{0\}$ when $f$ is not monomorphic. Moreover, the unique element $q(b)\in K$ satisfying~\eqref{eqn:equality for Schreier point} coincides with $q_{f(b),sf(b)}(b)$ using the notation of Section~\ref{section:Definition and first properties of Schreier extensions of monoids}. Conversely, if~\eqref{eqn:ses} is a Schreier extension in the sense of Definition~\ref{def:schreier}, which is split by a section $s\colon M\to X$ such that each $s(m)\in B_m,$ then $\xymatrix{{K} \ar@{>->}[r]^-k &X \ar@<.5ex>@{>>}[r]^-f &M \ar@<.5ex>[l]^-s}$ is a Schreier split extension. The Schreier retraction is given by the map of sets $q\colon X\dashrightarrow K,$ where $q(x)=q_{f(x),sf(x)}(x)$ for every $x\in X,$ as defined after Definition~\ref{def:schreier}.
\end{remark}

To avoid any confusion between the two notions, we shall stick to the name \emph{Schreier points}, from now on, to refer to~\eqref{eqn:schreier_point}. The relevance of Schreier points is that they correspond to monoid actions in the same way as split extensions of groups are equivalent to group actions, via the semidirect product construction (see~\cite[Proposition 5.2.2]{schreier_book}, and also Section 5.2 of~\cite{borceux-bourn} for the classical case of groups).

\begin{proposition}[{\cite[Theorem 5.1.2]{schreier_book}}]
\label{prop:actions}
Given a Schreier point $\xymatrix{{K} \ar@{>->}[r]_-k & B \ar@/_1pc/@{-->}[l]_-{q} \ar@<.5ex>@{>>}[r]^-f &M \ar@<.5ex>[l]^-s}$ with Schreier retraction $q,$ the induced monoid action of $M$ on $K$ is given by
\begin{equation}
\label{eqn:split_action}
M\times K\dashrightarrow K, \ (m,a)\mapsto m\bullet a=q\big(s(m)+k(a)\big).
\end{equation}
Note that~\eqref{eqn:split_action} is precisely~\eqref{eqn:action with q}, since each $s(m)$ is a representative of $m.$
\end{proposition}
We denote the monoid homomorphism corresponding to this action by
\begin{equation}
\label{eqn:Schreier point action sigma}
	\sigma\colon M \to \mathrm{End}(K), \ m \mapsto \sigma(m)\colon K \to K,
\end{equation}
where $\sigma(m)(a)=m\bullet a,$ as in~\eqref{eqn:split_action}.

Consider a reflexive graph $\xymatrix{{\mathcal{G}: \ R} \ar@<.9ex>[r]^-{\rho_1} \ar@<-.9ex>[r]_-{\rho_2} &{X} \ar[l]|-{\delta}}$ in $\mathbf{Mon},$ so that $\rho_1\delta=\rho_2\delta=id_X.$

\begin{definition}[{\cite[Definition 3.0.10]{schreier_book}}]
\label{def:S-graph}
We say that $\mathcal{G}$ is a \emph{Schreier reflexive graph}, or an \emph{$S$-reflexive graph}, if
\[
\xymatrix{{K_1} \ar@{>->}[r]^-{k_1} &R \ar@<.5ex>@{>>}[r]^-{\rho_1} &X \ar@<.5ex>[l]^-\delta}
\]
is a Schreier point, where $(K_1,k_1)$ is a kernel of $\rho_1.$
\end{definition}
Since any two kernels of $\rho_1$ are linked by a unique isomorphism of monoids, this does not depend on the choice of $(K_1,k_1).$ Of course, if $\mathcal{G}$ is a (reflexive) \emph{relation} on $X$ (meaning that $(\rho_1,\rho_2)$ are jointly monomorphic), we have a corresponding notion of an $S$-reflexive relation. We shall make use of the following result:
\begin{proposition}[{\cite[Proposition 3.1.5]{schreier_book}}]
\label{prop:3.1.5}
Any $S$-reflexive relation $\mathcal{G}$ is transitive, and it is symmetric if and only if $K_1$ is a group.
\end{proposition}

A relevant aspect here is that, in analogy with the case of Mal'tsev categories, there is a good notion of \emph{centrality} of ($S$-)equivalence relations. The study of centrality has a long story that goes back to Smith \cite{smith} for Mal'tsev varieties and to Carboni, Pedicchio, Pirovano \cite{carboni} and Pedicchio \cite{pedicchio} for the general context of Mal'tsev categories. Eventually, the notion of \emph{connector} was introduced in \cite{normality} and proven to be equivalent to the Smith-Pedicchio definition of centrality, given in terms of double centralizing relations (see also \cite{centrality}, Lemma $2.1$). Connectors can be defined very generally in any finitely complete category $\mathcal{C},$ but they behave particularly well in the context of Mal'tsev categories, where the uniqueness of the (eventual) connector between two equivalence relations $\mathcal{R}$ and $\mathcal{R}^\prime$ can be established.

Recall that if $\xymatrix{{\mathcal{R}: R} \ar@<.9ex>[r]^-{\rho_1} \ar@<-.9ex>[r]_-{\rho_2} &{X} \ar[l]|-{\delta}}$ and $\xymatrix{{\mathcal{R}^\prime: R^\prime} \ar@<.9ex>[r]^-{\rho_1^\prime} \ar@<-.9ex>[r]_-{\rho_2^\prime} &{X} \ar[l]|-{\delta^\prime}}$ are equivalence relations on an object $X$ in a finitely complete category $\mathcal{C},$ a connector between the two is a morphism ${p\colon R\times_XR^{\prime}\longrightarrow X}$ in $\mathcal{C}$ satisfying certain equational axioms (see \cite[Definition 2.2]{normality}), where $R\times_XR^\prime$ denotes a pullback
\[
\xymatrix{
{R\times_XR^\prime} \pullback \ar[d] \ar[r] &{R^\prime} \ar[d]^-{\rho_1^\prime}\\
R \ar[r]_-{\rho_2} &{X.}
}
\]
When $\mathcal{C}$ is a Mal'tsev category - meaning that any reflexive relation in $\mathcal{C}$ is an equivalence relation - these axioms come down to the two set-theoretical equations $p(a,a,z)=z$ and $p(b,y,y)=b$ for all $a,b,y,z\in X$ such that $a\mathcal{R}^\prime z$ and $b\mathcal{R}y.$ Moreover, if a connector between $\mathcal{R}$ and $\mathcal{R}^\prime$ exists, it is necessarily unique (\cite[Proposition 4.1]{normality}). Observe that, in the Mal'tsev context, a connector is simply a restricted Mal'tsev operation, where the restriction is to the triples $(x,y,z)$ such that $x\mathcal{R}y\mathcal{R}^\prime z.$

Now, the category of monoids is certainly not a Mal'tsev category, and even in the case of $S$-reflexive relations only transitivity, but not symmetry, is automatically guaranteed (Proposition~\ref{prop:3.1.5}). Nevertheless, it is a consequence of \cite[Theorem 2.4.2]{schreier_book} that when $\mathcal{R}$ and $\mathcal{R}^\prime$ are reflexive relations in $\mathbf{Mon}$ and \emph{at least one of the two is a Schreier reflexive relation}, the same reduced definition of a connector can be used to define centrality:
\begin{definition}[{\cite[Definition 4.3.1]{schreier_book}}]
\label{def:centralizing}
If $\xymatrix{{\mathcal{R}: R} \ar@<.9ex>[r]^-{\rho_1} \ar@<-.9ex>[r]_-{\rho_2} &{X} \ar[l]|-{\delta}}$ and $\xymatrix{{\mathcal{R}^\prime: R^\prime} \ar@<.9ex>[r]^-{\rho_1^\prime} \ar@<-.9ex>[r]_-{\rho_2^\prime} &{X} \ar[l]|-{\delta^\prime}}$ are reflexive relations in $\mathbf{Mon}$ and $\mathcal{R}^\prime$ is a Schreier reflexive relation, we say that $\mathcal{R}$ and $\mathcal{R}^\prime$ \emph{centralise each other} when there exists a monoid homomorphism $p\colon R\times_XR^\prime\longrightarrow X$ which satisfies $p(a,a,z)=z$ and $p(b,y,y)=b$ for all $a,b,y,z\in X$ such that $a\mathcal{R}^\prime z$ and $b\mathcal{R}y.$ The morphism $p$ is in this case necessarily unique, and it is called the \emph{connector} of $\mathcal{R}$ and $\mathcal{R}^\prime.$
\end{definition}

The following characterisation will be useful:
\begin{proposition}[{\cite[Proposition 4.3.3]{schreier_book}}]
\label{prop:char_centralizing}
Consider a reflexive relation $\mathcal{R}$ and an $S$-reflexive relation $\mathcal{R}^\prime$ as above, where
\[
\xymatrix{{K^\prime} \ar@{>->}[r]_-{k^\prime} &{R^\prime} \ar@/_1pc/@{-->}[l]_-{q^\prime} \ar@<.5ex>@{>>}[r]^-{\rho_1^\prime} &X \ar@<.5ex>[l]^-{\delta^\prime}}
\]
is a Schreier point with Schreier retraction $q^\prime.$ Then, considering $K^\prime\subseteq X$ and $k^\prime(t)=(0,t)$ for $t\in K^\prime,$ $\mathcal{R}$ and $\mathcal{R}^\prime$ centralise each other if and only if for every $t\in K^\prime$ and every $(x,y)\in R$ the equality $q^\prime(y,y + t) + x=x + t$ holds in $X;$ in this case, the connector is given by $p\big(x,y,z\big)=q^\prime(y,z) + x$ (for all $(x,y,z)\in X\times X\times X$ such that $x\mathcal{R}y\mathcal{R}^\prime z$).
\end{proposition}

These Mal'tsev aspects of Schreier internal structures have been explored at a categorical level in \cite{BMMS_S-protomod}.

\section{The Chasles relation of a cc-Schreier extension}
\label{sec:chasles}
Now we resume the study of extensions and fix a Schreier extension
\[ \xymatrix{{\mathbb{E}: K} \ar@{>->}[r]^-k &X \ar@{->>}[r]^-f &{M,}} \]
with $K$ commutative. 

Define a relation
\begin{equation}
\label{eqn:R_E}
 \xymatrix{\rel{E}: \ \R{E} \ar@<.9ex>[r]^-{r_1} \ar@<-.9ex>[r]_-{r_2} &{X} \ar[l]|-{s_0}}
\end{equation}
on $X$ in $\mathbf{Mon}$ by $(x,z)\in\R{E}$ if and only if $z=k(a)+x$ for some (not necessarily unique) $a\in K,$ with $r_1(x,z)=x,$ $r_2(x,z)=z$ and $s_0(x)=(x,x).$ Using~\eqref{eqn:Patrick}, we see that this is an internal relation in $\mathbf{Mon},$ for we can write
\begin{equation*}
\begin{split}
\big(x,k(a)+x\big)+\big(y,k(b)+y\big)&=\big(x+y,k(a)+x+k(b)+y\big)\\
&=\Big(x+y,k(a)+k\big(f(x)\ast b\big)+x+y\Big)\\
&=\Big(x+y,k\big(a+f(x)\ast b\big)+x+y\Big).
\end{split}
\end{equation*}
Moreover, $\rel{E}$ is clearly reflexive and transitive. Since $f\big(k(a)+x\big)=f(x)$ for all $x\in X$ and all $a\in K,$ there is an induced morphism $j$ in the kernel pair of $f$
\[
\xymatrix{
{\R{E}} \ar@/^1pc/[rrd]^-{r_2} \ar@/_1pc/[rdd]_-{r_1} \ar@{.>}[rd]|{j} &\ &\ \\
&{\mathrm{Eq}(f)} \pullback\ar[d]_-{f_1} \ar[r]^-{f_2}  &X \ar[d]^-f \\
&X \ar[r]_-f &{M};
}
\]
$j$ is a monomorphism by the equality $\la f_1,f_2\ra j=\la r_1,r_2\ra$ and it can be realised as an inclusion map. Thus $\R{E}$ is a subobject of $\mathrm{Eq}(f),$ and it follows from the next proposition that the kernel pair relation of $f,$ denoted by $\mathcal{E}q(f),$ is the congruence generated by the relation $\rel{E}$ (i.e., the smallest equivalence relation on $X$ in $\mathbf{Mon}$ containing $\rel{E}$):

\begin{proposition}
\label{prop:f_coeq}
Given a Schreier extension $\mathbb{E}$ with $K$ commutative and the relation $\rel{E}$~\eqref{eqn:R_E}, the regular epimorphism $f$ is a coequaliser of $(r_1,r_2).$
\end{proposition}

\begin{proof}
We have already argued that $fr_1=fr_2.$ Suppose that $h:X\rightarrow Z$ is a monoid homomorphism satisfying $hr_1=hr_2.$ Given $(x,y)\in \mathrm{Eq}(f),$ let $f(x)=f(y)=m.$ Then, $h(x)\stackrel{\eqref{eqn:x=kq(x)+u}}{=}h(kq(x)+u_m)=h(u_m),$ since $hr_1=hr_2$; similarly, $h(y)=h(u_m).$ We get $hf_1=hf_2,$ and $h$ factors uniquely through $f$ because $f$ is a coequaliser of its kernel pair.
\end{proof}

To give a more precise account of the situation, denote by $k_2=\la k,0\ra\colon K\rightarrowtail\mathrm{Eq}(f)$ the kernel of the second projection $f_2$ of the kernel pair $\mathrm{Eq}(f).$ Then we have:

\begin{proposition}
Given a Schreier extension $\mathbb{E}$ with $K$ commutative and the relation $\rel{E}$~\eqref{eqn:R_E}, the morphisms $j$ and $k_2$ are jointly extremal epimorphic.
\end{proposition}

\begin{proof}
For all $(x,y)\in\mathrm{Eq}(f)$ we can write
\begin{equation*}
\begin{split}
(x,y)\stackrel{\eqref{eqn:x=kq(x)+u}}{=}\big(kq(x)+u_m,kq(y)+u_m\big)&=(kq(x),0)+(u_m,kq(y)+u_m)\\
&=k_2(q(x))+j(u_m,kq(y)+u_m),
\end{split}
\end{equation*}
where $m=f(x)=f(y).$
\end{proof}

Hence the following characterisation of the Schreier extensions $\mathbb{E}$ for which $\rel{E}$ is a congruence:

\begin{corollary}
\label{cor:R_E-symmetric}
Given a Schreier extension $\mathbb{E}$ with $K$ commutative and the relation $\rel{E}$~\eqref{eqn:R_E}, the following statements are equivalent:
\begin{enumerate}
	\item[\em{(1)}] $k_2$ factors through $j;$
	\item[\em{(2)}] $j$ is an isomorphism;
	\item[\em{(3)}] $\rel{E}$ is symmetric;
	\item[\em{(4)}] $K$ is a group.
\end{enumerate}
\end{corollary}

\begin{proof}
$(1)\Rightarrow(2)$ follows from the previous proposition, considering the diagram
\[
\xymatrixcolsep{2.5pc}
\xymatrix{
&{\R{E}} \ar@{>->}[d]^-{j} &\ \\
K \ar@{>->}[r]_-{k_2} \ar@/^1pc/[ru] &{\mathrm{Eq}(f)} &{\R{E}.} \ar@{>->}[l]^-{j} \ar@/_1pc/@{=}[lu]
}
\]
$(2)\Rightarrow(3)$ is clear.\\
$(3)\Rightarrow(4)$ and $(4)\Rightarrow(1)$ come from the remark that an element $a\in K$ is invertible if and only if $\big(k(a),0\big)\in\R{E},$ whereas $\big(0,k(a)\big)\in\R{E}$ is always true.
\end{proof}

The fact that $a\in K$ is invertible if and only if $\big(k(a),0\big)\in\R{E}$ means in particular that we have a pullback square
\begin{equation*}
\xymatrix{
{U(K)} \pullback \ar@{>->}[r]^-i \ar@{>->}[d]_{\overline{k_2}} &K \ar@{>->}[d]^-{k_2} \\
{\R{E}} \ar@{>->}[r]_-j &{\mathrm{Eq}(f),}
}
\end{equation*}
where $i$ is the inclusion morphism and $\overline{k_2}(a)=\big(k(a),0\big).$

In the above notation, if $K$ is also cancellative, we can prove that $\rel{E}$
is an $S$-reflexive relation (Definition~\ref{def:S-graph}) and that it can be seen as a relation on the whole extension $\mathbb{E}.$

\begin{lemma}
\label{lemma:unique_a}
If $\mathbb{E}$ is a cc-Schreier extension as above and $(x,y)\in \R{E}$ (see~\eqref{eqn:R_E}), then $y=k(a)+x$ for a \emph{unique} $a\in K.$
\end{lemma}

\begin{proof}
Indeed, suppose that $k(a)+x=k(a^\prime)+x.$ Using~\eqref{eqn:x=kq(x)+u} we get $k(a)+kq(x)+u_{f(x)} = k(a^\prime)+kq(y)+u_{f(x)}.$ Applying Lemma~\ref{lemma:rep_invertible}$(1)$ and the cancellativity of $K,$ we conclude that $a=a^\prime.$
\end{proof}

\begin{proposition} If $\mathbb{E}$ is a cc-Schreier extension as above, then
\begin{equation}
\label{eqn:R_E-point}
\xymatrixcolsep{3pc}
\xymatrix{{K} \ar@{>->}[r]^-{k_1=\la 0,k \ra} &{\R{E}} \ar@<.5ex>@{>>}[r]^-{r_1} &X \ar@<.5ex>[l]^-{s_0}}
\end{equation}
is a Schreier point (see~\eqref{eqn:R_E}); consequently, $\rel{E}$ is an $S$-reflexive relation.
\end{proposition}
\begin{proof} It is easy to see that $k_1$ is the kernel of $r_1.$ For every $\big(x,k(a)+x\big)\in\R{E}$ we can write
\[
\big(x,k(a)+x\big)=\big(0,k(a)\big)+(x,x)=k_1(a)+s_0r_1\big(x,k(a)+x\big)
\]
for an element $a\in K$ which is unique by Lemma~\ref{lemma:unique_a}, and we conclude that~\eqref{eqn:R_E-point} is a Schreier point.
The corresponding Schreier retraction is the map of sets $q_1\colon \R{E} \dashrightarrow K,$ defined by $q_1(x,k(a)+x)=a.$
\end{proof}

We denote by $SExt_M$ the category whose objects are the Schreier extensions~\eqref{eqn:ses} with fixed codomain $M$ and whose morphisms are the morphisms~\eqref{eqn:morphism_of_ses} with $M^\prime=M$ and $\alpha_3=id_M.$ Since $\alpha_1$ is uniquely determined by $\alpha_2$ and the universal property of the kernel $(K^\prime,k^\prime)$ of $f',$ we denote such  morphisms as $(\alpha_1,\alpha_2)\colon \E \to \E',$ or simply as $\alpha_2\colon \E\to \E'.$

\begin{proposition}
If $\mathbb{E}$ is a cc-Schreier extension as above, then the relation $\rel{E}$ determines a relation
\begin{equation}
\begin{aligned}
\label{eqn:R_E-on-E}
\xymatrixcolsep{3.5pc}
\xymatrix{
{K\times K} \ar@<-.5ex>[d]_-{p_1} \ar@<.5ex>[d]^-{+} \ar@{>->}[r]^-{\hat{k}} &{\R{E}} \ar@<-.5ex>[d]_-{r_1} \ar@<.5ex>[d]^-{r_2} \ar@{->>}[r]^-{fr_1=fr_2} &{M} \ar@{=}[d] \\
K \ar@{>->}[r]_-k &X \ar@{->>}[r]_-f &M
}
\end{aligned}
\end{equation}
on $\mathbb{E}$ in $SExt_M,$ where $\hat{k}(a,b)=\big(k(a),k(b)+k(a)\big),$ $p_1(a,b)=a$ and $+$ is the monoid operation on $K.$
\end{proposition}

\begin{proof}
The fact that $\hat{k}$ is a kernel of $fr_1=fr_2$ is immediate, as is the fact that $+$ is a monoid homomorphism (because $K$ is commutative). It is also easy to see that, since $K$ is cancellative, the morphisms $p_1$ and $+$ are jointly monomorphic in $\mathbf{Mon}.$ For every $\big(x,k(b)+x\big)\in\R{E}$ we have $\big(x,k(b)+x\big)\stackrel{\eqref{eqn:x=kq(x)+u}}{=}\big(kq(x),k(b)+kq(x)\big)+\big(u_{f(x)},u_{f(x)}\big)=\hat{k}(q(x),b)+\big(u_{f(x)},u_{f(x)}\big)$ for the unique $q(x)\in K.$ The couple $(q(x),b)$ such that $\big(x,k(b)+x\big)=\hat{k}(q(x),b)+\big(u_{f(x)},u_{f(x)}\big)$ is then unique by the uniqueness of $q(x)$ and the cancellativity of $K$ (see Lemma~\ref{lemma:unique_a}). This proves that the upper row in~\eqref{eqn:R_E-on-E} is a Schreier extension with representatives $(u_m,u_m).$ It follows that $(p_1,r_1),(+,r_2)$ are morphisms in $SExt_M,$ and since both pairs $(p_1,+)$ and $(r_1,r_2)$ are jointly monomorphic, we conclude that~\eqref{eqn:R_E-on-E} is a relation on $\mathbb{E}.$
\end{proof}

The main point, now, is the following:

\begin{proposition}
If $\mathbb{E}$ is a cc-Schreier extension as above, the $S$-reflexive relation $\rel{E}$~\eqref{eqn:R_E} is self-centralizing in the sense of Definition~\ref{def:centralizing}.
\end{proposition}

\begin{proof}
By Proposition~\ref{prop:char_centralizing}, it is enough to prove that for every $(x,y)\in\R{E}$ and every $b\in K$ one has $kq_1\big(y,y+k(b)\big)+x=x+k(b),$ where $q_1\colon\rel{E}\dashrightarrow K$ is the Schreier retraction of~\eqref{eqn:R_E-point}. Now, if $(x,y)\in\R{E},$ we have $y=k(a)+x$ for a unique $a\in K$ (see Lemma~\ref{lemma:unique_a}), and
\begin{equation*}
\begin{split}
q_1\big(k(a)+x,k(a)+x+k(b)\big)&=q_1\big(k(a)+x,k(a)+k\big(f(x)\ast b\big)+x\big)\\
&=q_1\big(k(a)+x,k\big(f(x)\ast b\big)+k(a)+x\big)\\
&=f(x)\ast b,
\end{split}
\end{equation*}
using~\eqref{eqn:Patrick} and the commutativity of $K;$ then indeed
\[
kq_1\big(k(a)+x,k(a)+x+k(b)\big)+x=k\big(f(x)\ast b\big)+x=x+k(b),
\]
using again~\eqref{eqn:Patrick}.
\end{proof}

The same Proposition~\ref{prop:char_centralizing} gives us the explicit definition of the connector between $\rel{E}$ and itself as $p=p_{\mathbb{E}}\colon \R{E}\times_X\R{E}\longrightarrow X,$
\begin{equation}
\label{eqn: def of p_E}
	p\big(x,k(a)+x,k(b)+k(a)+x\big)=k(b)+x,
\end{equation}
for any cc-Schreier extension $\mathbb{E}.$ We stress the fact that the uniqueness of $p$ entails that this construction depends only on $\mathbb{E}.$

\begin{remark}
\label{rmk:group_slice}
If $X$ is a group, so that $\mathbb{E}$ is a short exact sequence of groups with abelian kernel $K,$ by Corollary~\ref{cor:R_E-symmetric} the relation $\rel{E}$ is simply $\mathcal{E}q(f),$ the kernel pair relation of $f,$ and by the above equation the connector $p\colon\mathrm{Eq}(f)\times_X\mathrm{Eq}(f)\cong\{(x,y,z)\in X\times X\times X:f(x)=f(y)=f(z)\}\longrightarrow X$ is given by
\begin{equation*}
\begin{split}
p(x,y,z)&=p\big(x,k(a)+x,k(b)+k(a)+x\big)\\
&=k(b)+x\\
&=x-x-k(a)+k(b)+k(a)+x\\
&=x-\big(k(a)+x\big)+\big(k(b)+k(a)+x\big)\\
&=x-y+z
\end{split}
\end{equation*}
(recalling that here we can write $y=k(a)+x$ and $z=k(b)+y$ for unique elements $a,b\in K).$

This means that in this case $p$ is precisely the (unique, autonomous \cite{bourn-direction}) Mal'tsev operation on the internal Mal'tsev algebra $f$ in the slice category $\mathbf{Gp}/M$ (i.e. an object endowed with an internal Mal'tsev operation in $\mathbf{Gp}/M$).
\end{remark}

We shall denote for the sake of brevity $\P{E}=\R{E}\times_X\R{E}\cong\big\{\big(x,k(a)+x,k(b)+k(a)+x\big):x\in X \ \& \ a,b\in K\big\},$ so that we have a pullback
\begin{equation*}
\xymatrix{
{\P{E}} \pullback \ar[r]^-{p_2} \ar[d]_-{p_1}  &{\R{E}} \ar[d]^-{r_1} \\
{\R{E}} \ar[r]_{r_2}  &X
}
\end{equation*}
with projections $p_1\big(x,k(a)+x,k(b)+k(a)+x\big)=\big(x,k(a)+x\big)$ and $p_2\big(x,k(a)+x,k(b)+k(a)+x\big)=\big(k(a)+x,k(b)+k(a)+x\big).$

Now we come to the main core of the construction. For a cc-Schreier extension $\mathbb{E},$ define a relation $\Ch{E}$ on $\R{E}$ by: given $(x,k(a)+x), (y,k(b)+y)\in \R{E},$ let
\[
\big(x,k(a)+x\big)\Ch{E}\big(y,k(b)+y\big)\Longleftrightarrow\big(x,y,k(b)+y\big)\in \P{E} \ \& \ k(a)+x=p\big(x,y,k(b)+y\big),
\]
which is equivalent to
\[
\big(x,k(a)+x\big)\Ch{E}\big(y,k(b)+y\big)\Longleftrightarrow(x,y)\in\R{E} \ \& \ k(a)+x=k(b)+x
\]
and to
\begin{equation}
\label{eqn:ch_3}
\big(x,k(a)+x\big)\Ch{E}\big(y,k(b)+y\big)\Longleftrightarrow(x,y)\in\R{E} \ \& \ a=b
\end{equation}
(using Lemma~\ref{lemma:unique_a}). It is an internal relation in $\mathbf{Mon},$ because so is $\R{E}$ and $p$ is a monoid homomorphism, and it is clearly reflexive and transitive (using~\eqref{eqn:ch_3}). As a reflexive graph, it is represented by
\begin{equation}
\label{eqn:chasles}
\xymatrixcolsep{4pc}
\xymatrix{\Ch{E}: \P{E} \ar@<.9ex>[r]^-{\pi_1=\la r_1p_1,p\ra} \ar@<-.9ex>[r]_-{p_2} &{\R{E},} \ar[l]|-{\sigma_0}}
\end{equation}
where $\pi_1=\la r_1p_1,p\ra\colon\big(x,k(a)+x,k(b)+k(a)+x\big)\mapsto\big(x,k(b)+x\big)$ and $\sigma_0\colon\big(x,k(a)+x\big)\mapsto\big(x,x,k(a)+x\big).$

We shall call $\Ch{E}$ the \emph{Chasles relation} associated with the cc-Schreier extension $\mathbb{E},$ as it comes from Remark~\ref{rmk:group_slice} that it is indeed a generalisation of the homonymous relation defined in \cite{bourn-direction}.

As for $\rel{E}$ in the cancellative case, \eqref{eqn:chasles} turns out to be an $S$-reflexive relation:

\begin{proposition}
If $\mathbb{E}$ is a cc-Schreier extension as above and $\mathcal{R}_\E$ is the relation~\eqref{eqn:R_E}, then the split extension
\[
	\xymatrixcolsep{3.5pc}\xymatrix{{K} \ar@{>->}[r]^-{\kappa_1=\la 0,k,k\ra} & {\P{E}} \ar@<.5ex>@{>>}[r]^-{\pi_1} &{\R{E}} \ar@<.5ex>[l]^-{\sigma_0}}
\]
is a Schreier point; consequently, $\Ch{E}$ is an $S$-reflexive relation.
\end{proposition}

\begin{proof}
It is easy to see that $\kappa_1$ is the kernel of $\pi_1.$ For every $\big(x,k(a)+x,k(b)+k(a)+x\big)\in\P{E}$ we can write
\begin{equation*}
\begin{split}
\big(x,k(a)+x,k(b)+k(a)+x\big)&=\big(x,k(a)+x,k(a)+k(b)+x\big)\\
&=\big(0,k(a),k(a)\big)+\big(x,x,k(b)+x\big)\\
&=\kappa_1(a)+\sigma_0\big(x,k(b)+x\big)\\
&=\kappa_1(a)+\sigma_0\pi_1\big(x,k(a)+x,k(b)+k(a)+x\big),
\end{split}
\end{equation*}
with $a$ unique by Lemma~\ref{lemma:unique_a}. The corresponding Schreier retraction is the map of sets $\overline{q}_1\colon \P{E}\dashrightarrow K,$ defined by $\overline{q}_1\big(x,k(a)+x,k(b)+k(a)+x\big)=a.$
\end{proof}

By Proposition~\ref{prop:3.1.5} and Corollary~\ref{cor:R_E-symmetric}, then, we have:
\begin{corollary}
Given a cc-Schreier extension $\mathbb{E}$ as above, the Chasles relation $\Ch{E}$~\eqref{eqn:chasles} is symmetric if and only if $\R{E}$ is symmetric (if and only if $K$ is a group).
\end{corollary}

\section{Definition of the direction functor for cc-Schreier extensions}
\label{sec:direction}
Fix a cc-Schreier extension of monoids $\xymatrix{{\mathbb{E}: \ K} \ar@{>->}[r]^-k &X \ar@{->>}[r]^-f &M}$ and consider the Chasles relation $\Ch{E}$ associated with $\mathbb{E}$ as in~\eqref{eqn:chasles}, using the notation of the previous section. Let $\mathfrak{s}_0$ denote the unique morphism induced by the universal property of the pullback $(\P{E},p_1,p_2)$ as in the diagram
\begin{equation*}
\xymatrix{
{\R{E}} \ar@/^1pc/[rrd]^-{s_0r_2} \ar@/_1pc/@{=}[rdd] \ar@{.>}[rd]|{\mathfrak{s}_0} &\ &\ \\
&{\P{E}} \ar[r]^-{p_2} \ar[d]_-{p_1} \pullback &{\R{E}} \ar@<-.5ex>[d]_-{r_1} \\
&{\R{E}} \ar[r]_{r_2}  &{X,} \ar@<-.5ex>[u]_-{s_0}
}
\end{equation*}
$\mathfrak{s}_0(x,k(a)+x)=(x,k(a)+x,k(a)+x).$ It is easily seen that $\xymatrixcolsep{4pc}\xymatrix{{\P{E}} \ar@<.9ex>[r]^-{p_1} \ar@<-.9ex>[r]_-{\pi_2=\la p,r_2p_2 \ra} &{\R{E}} \ar[l]|-{\mathfrak{s}_0}}$ is also an $S$-reflexive relation on $\R{E}$ in $\mathbf{Mon}$ sharing all the properties of $\Ch{E},$ and that the diagram
\[
\xymatrixcolsep{3pc}
\xymatrixrowsep{3pc}
\xymatrix{
{\P{E}} \ar@<-.9ex>[d]_-{p_1} \ar@<.9ex>[d]^-{\pi_2} \ar@<.9ex>[r]^-{\pi_1} \ar@<-.9ex>[r]_-{p_2} &{\R{E}} \ar[l]|{\sigma_0} \ar@<-.9ex>[d]_-{r_1} \ar@<.9ex>[d]^-{r_2} \\
{\R{E}} \ar[u]|{\mathfrak{s}_0} \ar@<.9ex>[r]^-{r_1} \ar@<-.9ex>[r]_-{r_2} &{X} \ar[l]|{s_0} \ar[u]|{s_0}
}
\]
is commutative in the usual sense.

We know by Proposition~\ref{prop:f_coeq} that $f$ is a coequaliser of $\rel{E}$ and, as $\mathbf{Mon}$ is cocomplete, we can also consider a coequaliser of the Chasles relation $\Ch{E},$ which we denote by
\[
\xymatrix{{\gamma=\gamma_{\mathbb{E}}=coeq(\pi_1,p_2)\colon \R{E}}\ar@{->>}[r] &{df;}}
\]
thus
\begin{equation}
\label{eqn: def of gamma}
	\gamma(x,k(b)+x)=\gamma(k(a)+x,k(a)+k(b)+x), \;\;\forall a,b\in K.
\end{equation}

We get a commutative diagram
\begin{equation}
\begin{aligned}
\label{eqn:direction_D}
\xymatrixcolsep{3pc}
\xymatrixrowsep{3pc}
\xymatrix{
{\P{E}} \ar@<-.9ex>[d]_-{p_1} \ar@<.9ex>[d]^-{\pi_2} \ar@<.9ex>[r]^-{\pi_1} \ar@<-.9ex>[r]_-{p_2} &{\R{E}} \ar[l]|{\sigma_0} \ar@<-.9ex>[d]_-{r_1} \ar@<.9ex>[d]^-{r_2} \ar@{->>}[r]^-{\gamma} &{df} \ar@<-.5ex>[d]_-{\overline{f}}\\
{\R{E}} \ar[u]|{\mathfrak{s}_0} \ar@<.9ex>[r]^-{r_1} \ar@<-.9ex>[r]_-{r_2} &{X} \ar[l]|{s_0} \ar[u]|{s_0} \ar@{->>}[r]_-f &M \ar@<-.5ex>[u]_-{\overline{s}}
}
\end{aligned}
\end{equation}
in $\mathbf{Mon},$ where $\overline{f}$ and $\overline{s}$ are uniquely induced by the universal property of the coequalisers $\gamma$ and $f,$ respectively, and are accordingly (well) defined by
\begin{equation}
\label{eqn: overline f and s}
	\overline{f}\big(\gamma(x,k(a)+x)\big)=f(x)\;\; \text{and} \;\;\overline{s}(m)=\gamma(x,x),
\end{equation}
where $x\in X$ is any element satisfying $f(x)=m.$ Clearly one has $\overline{f}\overline{s}=id_M,$ and it also follows from the commutativity of the upward right-hand square in~\eqref{eqn:direction_D} that
\begin{equation}
\label{eqn:equal gammas}
	\gamma(x,x)=\gamma(y,y)\;\; \text{for all } (x,y)\in\mathrm{Eq}(f).
\end{equation}

Note that $\Ch{E}$ is not always a congruence (it is not symmetric if $K$ is not a group), so that in general we only have an inclusion $\Ch{E}\subseteq \mathcal{E}q(\gamma).$  As a consequence, the implication $\gamma\big(x,k(a)+x\big)=\gamma\big(y,k(b)+y\big)\Rightarrow \big(x,k(a)+x\big)\Ch{E}\big(y,k(b)+y\big)$ is, in general, false. However, we are able to describe explicitly the congruence generated by the Chasles relation, i.e. the kernel pair relation $\mathcal{E}q(\gamma)$ of the coequaliser $\gamma$ of $\Ch{E}$:

\begin{lemma}
Consider the diagram~\eqref{eqn:direction_D}. The kernel pair relation of $\gamma$ coincides with the relation $\xymatrix{\tau: \ T \ar@<.5ex>[r]^-{t_1} \ar@<-.5ex>[r]_-{t_2} & {\R{E}}}$ defined by $\Big( \big(x,k(a)+x\big), \big(y,k(b)+y\big)\Big)\in T$ if and only if $a=b$ and $(x,y)\in \mathrm{Eq}(f)$; $t_1$ and $t_2$ are the first and second projections, respectively.
\end{lemma}
\begin{proof}
By Lemma~\ref{lemma:unique_a}, $\tau$ is well defined. Suppose now that $\Big( \big(x,k(a)+x\big), \big(y,k(a)+y\big)\Big) \in T$ and $\Big(\big(x^\prime,k(a^\prime)+x^\prime\big),\big(y^\prime,k(a^\prime)+y^\prime\big)\Big) \in T.$ Using~\eqref{eqn:Patrick}, we get
\begin{equation*}
\begin{split}
\big(x,k(a)+x\big)+\big(x^\prime,k(a^\prime)+x^\prime\big)&=\big(x+x^\prime,k(a)+x+k(a^\prime)+x^\prime\big)\\
&=\big(x+x^\prime,k(a)+k(f(x)\ast a^\prime)+x+x^\prime\big)
\end{split}
\end{equation*}
and similarly $\big(y,k(a)+y\big)+\big(y^\prime,k(a^\prime)+y^\prime\big)=\big(y+y^\prime,k(a)+k(f(y)\ast a^\prime)+y+y^\prime\big),$ with $k(a)+k(f(x)\ast a^\prime)=k(a)+k(f(y)\ast a^\prime)$; it is obvious that $(x+x',y+y')\in \mathrm{Eq}(f).$ This proves that $\tau$ is an internal relation on $\R{E}$ in $\mathbf{Mon}.$  It is clear that $\tau$ is reflexive, symmetric, and transitive: thus $\tau$ is a congruence on $\R{E}.$

It is also easy to check that $\big(x,k(a)+x\big)\Ch{E}\big(k(a^\prime)+x,k(a)+k(a^\prime)+x\big)$ implies $\big(x,k(a)+x\big)\tau\big(k(a^\prime)+x,k(a)+k(a^\prime)+x\big)$; it follows that $\Ch{E}\subseteq\tau,$ and thus $\mathcal{E}q(\gamma)\subseteq\tau$ by minimality of $\mathcal{E}q(\gamma)$ among all congruences on $\R{E}$ containing $\Ch{E}.$ Conversely, consider $\big(x,k(a)+x\big)\tau\big(y,k(a)+y\big).$ Suppose that $f(x)=f(y)=m$ and let $u_m\in B_m.$ We use~\eqref{eqn:x=kq(x)+u}, \eqref{eqn:equal gammas} and the commutativity of $K$ to get
\begin{equation*}
\begin{split}
\gamma\big(x,k(a)+x\big)&=\gamma\big(kq(x)+u_m,k(a)+kq(x)+u_m\big)\\
&=\gamma\big(kq(y)+kq(x)+u_m,kq(y)+k(a)+kq(x)+u_m\big)\\
&=\gamma\big(kq(x)+kq(y)+u_m,k(a)+kq(x)+kq(y)+u_m\big)\\
&=\gamma\big(kq(x)+y,kq(x)+k(a)+y\big)\\
&=\gamma\big(y,k(a)+y\big).
 \end{split}
\end{equation*}
This gives $\tau\subseteq\mathcal{E}q(\gamma),$ so that $\tau= \mathcal{E}q(f).$
\end{proof}

\begin{corollary}
\label{cor:a_equals_b}
Consider $\gamma$ as in~\eqref{eqn:direction_D}. We have:
\begin{enumerate}
	\item[\em{(1)}] $\gamma\big(x,k(a)+x\big)=\gamma\big(y,k(b)+y\big)$ if and only if $a=b$ and $(x,y)\in \mathrm{Eq}(f)$;
	\item[\em{(2)}] if $(x,y)\in \mathrm{Eq}(f),$ then $\gamma\big(x,k(a)+x\big)=\gamma\big(y,k(a)+y\big)$;
\end{enumerate}
hence in particular
\begin{enumerate}
	\item[\em{(3)}] $\gamma\big(x,k(a)+x\big)=\gamma\big(x,k(b)+x\big)$ if and only if $a=b.$
\end{enumerate}
\end{corollary}

\begin{proposition}
\label{prop:df_schreier}
Given a cc-Schreier extension $\E$ as above, the split extension
\begin{equation}
\label{eqn:Schreier point overline f}
\xymatrix{{K} \ar@{>->}[r]_-{\overline{\kappa}} & {df} \ar@/_1pc/@{-->}[l]_-{\overline{q}} \ar@<.5ex>@{>>}[r]^-{\overline{f}} &M, \ar@<.5ex>[l]^-{\overline{s}}}
\end{equation}
where $\overline{\kappa}(a)=\gamma\big(0,k(a)\big),$ is a Schreier point with Schreier retraction defined by $\overline{q}(\gamma\big(x,k(a)+x\big))=a.$
\end{proposition}

\begin{proof}
It follows from Corollary~\ref{cor:a_equals_b}(3) that $\overline{\kappa}$ in a monomorphism. Moreover, it is easy to check that $\overline{f}\overline{\kappa}=0,$ and if $\gamma(x,k(a)+x)\in df$ is such that $\overline{f}(\gamma(x,k(a)+x))=1$ then $f(x)=1=f(0)$ (see~\eqref{eqn: overline f and s}): thus, by Corollary~\ref{cor:a_equals_b}(2), we have $\gamma(x,k(a)+x)=\gamma(0,k(a))=\overline{\kappa}(a).$

The function $\overline{q}$ is well defined by Corollary~\ref{cor:a_equals_b}(1). For every $\big(x,k(a)+x\big)\in\R{E}$ there is a unique decomposition $\big(x,k(a)+x\big)=\big(0,k(a)\big)+\big(x,x\big),$ so that $\gamma\big(x,k(a)+x\big)=\gamma\big(0,k(a)\big)+\gamma\big(x,x\big)=\overline{\kappa}(a)+\overline{s}\overline{f}\big(\gamma\big(x,k(a)+x\big)\big)$ (using the fact that $\gamma$ is a monoid homomorphism and~\eqref{eqn: overline f and s}). The uniqueness of $a$ in $\gamma\big(x,k(a)+x\big)=\overline{\kappa}(a)+\overline{s}\overline{f}\big(\gamma\big(x,k(a)+x\big)\big)$ is the content of Corollary~\ref{cor:a_equals_b}(3).
\end{proof}

Denote by $SPt_M$ the category having as objects all Schreier points~\eqref{eqn:schreier_point}, with $M$ fixed (we shall say that~\eqref{eqn:schreier_point} is a Schreier point \emph{on M}), and as morphisms the morphisms of split exact sequences
\[
\xymatrix{
\mathbb{S}: {K} \ar@<2ex>[d]_-{\lambda_1} \ar@{>->}[r]^-{k} &{B} \ar@/_1pc/@{-->}[l]_-{q} \ar[d]_-{\lambda} \ar@<.5ex>@{>>}[r]^-{f} &M \ar@{=}[d] \ar@<.5ex>[l]^-{s} \\
\mathbb{S}': {K^\prime} \ar@{>->}[r]_-{k^\prime} &{B^\prime} \ar@/_1pc/@{-->}[l]_-{q^\prime} \ar@<.5ex>@{>>}[r]^-{f^\prime} &{M} \ar@<.5ex>[l]^-{s^\prime}
}
\]
in $\mathbf{Mon},$ i.e., $\lambda_1,\lambda$ are monoid homomorphisms such that $k^\prime\lambda_1=\lambda k,$ $f^\prime\lambda=f$ and $\lambda s=s^\prime.$ As observed in \cite[Proposition 2.3.1]{schreier_book}, it follows that $\lambda_1q=q^\prime \lambda,$ namely the morphism $(\lambda_1,\lambda)$ preserves the Schreier retraction. As in the case of the morphisms in $SExt_M,$ we shall denote these as $(\lambda_1,\lambda)\colon\mathbb{S} \to \mathbb{S}',$ or simply as $\lambda\colon \mathbb{S}\to \mathbb{S}'.$

By the previous proposition, $(\overline{f},\overline{s})$ is an object in $SPt_M$ and $\overline{f}$ has a commutative and cancellative kernel $K;$ the next result allows us to understand how these objects relate to the Schreier extensions~\eqref{eqn:ses}.

Denote by $c\text{-}SExt_M$ (resp., $cc\text{-}SExt_M$) the full subcategory of $SExt_M$ of all Schreier extensions with codomain $M$ and commutative (resp., commutative and cancellative) kernel, and similarly by $c\text{-}SPt_M$ (resp., $cc\text{-}SPt_M$) the full subcategory of $SPt_M$ whose objects are the Schreier points on $M$ having a commutative (resp., commutative and cancellative) kernel.

Now, recall the following classical definition:
\begin{definition}
\label{dfn:int_comm_monoid}
Given a category $\mathcal{C}$ with finite products, we say that an object $X$ of $\mathcal{C}$ has an \emph{internal commutative monoid structure} in $\mathcal{C}$ when there exist morphisms $\omega\colon X\times X \to X$ (called the \emph{multiplication}) and $\varepsilon\colon 1\to X$ (called the \emph{unit}) such that
\[
	\begin{array}{ll}
		\omega(\omega\times id_X)=\omega(id_X\times \omega) & \text{associativity axiom}, \\
		\omega\la p_2,p_1\ra =\omega & \text{commutativity axiom}, \\
		\omega\la \varepsilon!_X,id_X\ra=id_X & \text{unit axiom}.
\end{array}
\]
(Here $!_X$ stands for the terminal morphism $X\rightarrow 1.$)

A morphism of internal commutative monoids between two internal commutative monoids $(X,\omega,\varepsilon)$ and $(X',\omega',\varepsilon')$ is a morphism $f\colon X\to X'$ in $\mathcal{C}$ such that $f\varepsilon =\varepsilon'$ ($f$ preserves the unit) and $f\omega=\omega'(f\times f)$ ($f$ preserves the multiplication). We denote by $\mathbf{CMon}(\mathcal{C})$ the category of internal commutative monoids in $\mathcal{C}.$
\end{definition}
Then we have:
\begin{theorem}
\label{thm:internal_monoids}
There are equivalences of categories $c\text{-}SPt_M$ $\cong$ $\mathbf{CMon}(SExt_M)$ $\cong$ $\mathbf{CMon}(c\text{-}SExt_M).$ When the cancellativity of the kernel is considered, this equivalence restricts to an equivalence $cc\text{-}SPt_M$ $\cong$ $\mathbf{CMon}(cc\text{-}SExt_M).$
\end{theorem}

In order to prove the theorem, we shall need the following result, which appears as Proposition $2.1.5$ in \cite{schreier_book} and which is the version for Schreier points of our Lemma~\ref{prop:2.1.5}:
\begin{lemma}
\label{lemma:original}
Given a Schreier point $\xymatrix{{K} \ar@{>->}[r]_-k &B \ar@/_1pc/@{-->}[l]_-{q} \ar@<.5ex>@{>>}[r]^-f &M \ar@<.5ex>[l]^-s}$ with Schreier retraction $q,$ the following properties hold true:
\begin{enumerate}
	\item[\em{(1)}] $qk=id_K;$
	\item[\em{(2)}] $qs=0;$
	\item[\em{(3)}] $kq\big(s(m)+k(a)\big)+s(m)=s(m)+k(a)$ for all $m\in M$ and $a\in K;$
	\item[\em{(4)}] $q(b+b^\prime)=q(b)+q\big(sf(b)+kq(b^\prime)\big)$ for all $b,b^\prime\in B.$
\end{enumerate}
\end{lemma}

\begin{proof}[Proof of Theorem~\ref{thm:internal_monoids}]
Let $\xymatrix{{K} \ar@{>->}[r]_-k &B \ar@/_1pc/@{-->}[l]_-{q} \ar@<.5ex>@{>>}[r]^-f &M \ar@<.5ex>[l]^-s}$ be a Schreier point on $M$ with $K$ commutative. We want to prove that there is a natural commutative monoid structure on the extension $\xymatrix{{\mathbb{E}:K} \ar@{>->}[r]^-k &B \ar@{>>}[r]^-f &M}$ in $(c\text{-})SExt_M.$ The product $\mathbb{E}\times\mathbb{E}$ in $(c\text{-})SExt_M$ can be realised by
\[
\xymatrixcolsep{3pc}
\xymatrix{
\mathbb{E}\times \mathbb{E}: {K\times K} \ar@<3.5ex>[d]_-{p_1} \ar@<4.5ex>[d]^-{p_2} \ar@{>->}[r]^-{k\times_M k} &{\mathrm{Eq}(f)} \ar@<-.5ex>[d]_-{f_1} \ar@<.5ex>[d]^-{f_2} \ar@{>>}[r]^-{ff_1=ff_2} &M \ar@{=}[d] \\
\;\;\;\;\;\;\mathbb{E}:{K} \ar@{>->}[r]_-k &B \ar@{>>}[r]_-f &{M,}
}
\]
where $\big(\mathrm{Eq}(f),f_1,f_2\big)$ is a kernel pair of $f$ and $p_1,p_2$ are the obvious projections. It is easy to check that $\mathbb{E}\times\mathbb{E}$ is indeed a Schreier extension, with representatives of the type $(u_m,v_m)\in B_m\times B_m,$ for every $m\in M$ (here $B_m$ denotes the set of representatives of $m$ for $\mathbb{E},$ as usual).

We define a map $\mu\colon\mathrm{Eq}(f)\longrightarrow B$ by $\mu(x,y)=kq(x)+y.$ It is a monoid homomorphism, because $\mu(0,0)=0,$ $\mu(x, y)+\mu(x^\prime,y^\prime)=kq(x)+y+kq(x^\prime)+y^\prime,$ and
\begin{equation*}
\begin{split}
\mu\big((x,y)+(x^\prime,y^\prime)\big)&=\mu\big(x+x^\prime,y+y^\prime\big)\\
&=kq(x+x^\prime)+y+y^\prime\\
&=kq(x)+kq\big(sf(x)+kq(x^\prime)\big)+kq(y)+sf(y)+y^\prime \\
&=kq(x)+kq(y)+kq\big(sf(y)+kq(x^\prime)\big)+sf(y)+y^\prime\\
&=kq(x)+kq(y)+sf(y)+kq(x^\prime)+y^\prime\\
&=kq(x)+y+kq(x^\prime)+y^\prime,
\end{split}
\end{equation*}
using Lemma~\ref{lemma:original}(4), Lemma~\ref{lemma:original}(3), \eqref{eqn:equality for Schreier point} and the commutativity of $K.$

We claim that
\begin{equation}
\begin{aligned}
\label{eqn:mu}
\xymatrixcolsep{3pc}
\xymatrix{
\mathbb{E}\times \mathbb{E}: {K\times K} \ar@{}[rd]|{\mathrm{(a)}} \ar@<4ex>[d]_-{+} \ar@{>->}[r]^-{k\times_M k} &{\mathrm{Eq}(f)} \ar@{}[rd]|{\mathrm{(b)}} \ar[d]^-{\mu} \ar@{>>}[r]^-{ff_1=ff_2} &M \ar@{=}[d] \\
\;\;\;\;\;\;\; \mathbb{E}: {K} \ar@{>->}[r]_-k &B \ar@{>>}[r]_-f &{M,}
}
\end{aligned}
\end{equation}
where $+$ denotes the monoid operation on $K$ (which is a monoid homomorphism because $K$ is commutative), is a morphism of Schreier extensions. The commutativity of $\mathrm{(b)}$ is immediate by the definition of $\mu,$ and $\mathrm{(a)}$ commutes by Lemma~\ref{lemma:original}(1). To prove that $\mu$ preserves the representatives, let $(u_m,v_m)\in B_m\times B_m$ be a representative of $m\in M$ for $\E\times \E.$ We want to show that $\mu(u_m,v_m)=kq(u_m)+v_m\in B_m,$ i.e. that for any $b\in f^{-1}(m)$ one can write $b=k(a)+kq(u_m)+v_m$ for a unique $a\in K.$ As $u_m$ and $v_m$ are representatives of $m,$ we have $b=k(a^\prime)+u_m$ and $s(m)=k(a^{\prime\prime})+v_m$ for unique elements $a^\prime, a^{\prime\prime}\in K,$ so that
\begin{equation*}
\begin{split}
b = k(a^\prime)+u_m& \stackrel{\eqref{eqn:equality for Schreier point}}{=} k(a^\prime)+kq(u_m)+s(m)\\
&=k(a^\prime)+kq(u_m)+k(a^{\prime\prime})+v_m\\
&=k(a^\prime+a^{\prime\prime})+kq(u_m)+v_m \ \ (K \text{ is commutative})\\
&=k(a)+kq(u_m)+v_m
\end{split}
\end{equation*}
with $a=a^\prime+a^{\prime\prime}.$ Eventually, if $k(a)+kq(u_m)+v_m=k(\overline{a})+kq(u_m)+v_m$ for some other $\overline{a}\in K,$ then $a+q(u_m)=\overline{a}+q(u_m)$ by Lemma~\ref{lemma:rep_invertible}(1), and hence $a=\overline{a}$ because $q(u_m)$ is invertible in $K,$ by Lemma~\ref{lemma:rep_invertible}(2) (as $u_m=kq(u_m)+s(m)$ and both $u_m$ and $s(m)$ are representatives of $m$).

We now prove that the morphism $\mu\colon \E\times \E\to \E$ gives a multiplication on $\E$ in $(c\text{-})SExt_M.$ First, $\mu$ is associative and commutative. Indeed, denote $\mu(x,y)=x\cdot y$ for all $(x,y)\in\mathrm{Eq}(f);$ then we have:
\begin{itemize}
\item \emph{Associativity}: $x\cdot(y\cdot z)=kq(x)+\big(kq(y)+z\big)$ and $(x\cdot y)\cdot z=kq\big(kq(x)+y\big)+z=\big(kq(x)+kq(y)\big)+z$ (by Lemma~\ref{lemma:original});
\item \emph{Commutativity}: $x\cdot y=kq(x)+y \stackrel{\eqref{eqn:equality for Schreier point}}{=} kq(x)+kq(y)+sf(y) = kq(y)+kq(x)+sf(x) \stackrel{\eqref{eqn:equality for Schreier point}}{=} kq(y)+x=y\cdot x$ (by the commutativity of $K$ and the fact that $f(x)=f(y)$).
\end{itemize}

The unit on the extension $\E$ is given by $s\colon \mathds{1}\to \E$
\begin{equation}
\begin{aligned}
\label{eqn:unit}
\xymatrix{
\mathds{1}: {0} \ar@<1.8ex>[d]_-{0_K} \ar@{>->}[r] &{M} \ar[d]^-s \ar@{=}[r] &M \ar@{=}[d] \\
\E: {K} \ar@{>->}[r]_-k &B \ar@{>>}[r]_-f &{M}
}
\end{aligned}
\end{equation}
(observe that $\xymatrix{{0} \ar@{>->}[r] &{M} \ar@{=}[r] &M}$ is a terminal object in $(c\text{-})SExt_M$ and that $s$ preserves the representatives $\{m\}$ of $m$ for $\mathds{1}$; cf. Example~\ref{ex:schreier_extensions}.2).

Indeed, the section $s$ of $f$ acts as a neutral element for $\mu$: for any $b\in B,$ $\mu\la sf,id_B\ra(b)=\mu(sf(b),b)=kqsf(b)+b=0+b=b,$ by applying Lemma~\ref{lemma:original} (see the unit axiom in Definition~\ref{dfn:int_comm_monoid}). 

We conclude that $\mathbb{E}$ is an internal commutative monoid in $(c\text{-})SExt_M$ with multiplication~\eqref{eqn:mu} and unit~\eqref{eqn:unit}.

The functoriality of this construction is readily established by observing that if
\[
\xymatrix{
{K} \ar[d]_-{\lambda_1} \ar@{>->}[r]_-{k} &{B} \ar@/_1pc/@{-->}[l]_-{q} \ar[d]_-{\lambda} \ar@<.5ex>@{>>}[r]^-{f} &M \ar@{=}[d] \ar@<.5ex>[l]^-{s} \\
{K^\prime} \ar@{>->}[r]_-{k^\prime} &{B^\prime} \ar@/_1pc/@{-->}[l]_-{q^\prime} \ar@<.5ex>@{>>}[r]^-{f^\prime} &{M} \ar@<.5ex>[l]^-{s^\prime}
}
\]
is a morphism of Schreier points, then
\[
\xymatrix{
{K} \ar[d]_-{\lambda_1} \ar@{>->}[r]^-{k} &{B} \ar[d]_-{\lambda} \ar@{>>}[r]^-{f} &M \ar@{=}[d] \\
{K^\prime} \ar@{>->}[r]_-{k^\prime} &{B^\prime} \ar@{>>}[r]_-{f^\prime} &{M}
}
\]
is a morphism of Schreier extensions (by Proposition~\ref{prop:rep_pres}, since $\lambda s=s^\prime$), which is a morphism in $\mathbf{CMon}((c\text{-})SExt_M)$ with respect to the monoid structures that we have defined above. Indeed, $\lambda s=s^\prime,$ so that the unit is preserved, and for all $(x,y)\in\mathrm{Eq}(f)$
\begin{equation*}
\begin{split}
\lambda\big(\mu(x,y)\big)&=\lambda\big(kq(x)+y\big)\\
&=k^\prime \lambda_1q(x)+\lambda(y)\\
&=k^\prime q^\prime\lambda(x)+\lambda(y)\\
&=\mu^\prime\big(\lambda(x),\lambda(y)\big) \\
&=\mu^\prime (\lambda\times_M \lambda)(x,y),
\end{split}
\end{equation*}
so that the multiplication is preserved; cf.  Definition~\ref{dfn:int_comm_monoid}.

Conversely, suppose that $\xymatrix{{\mathbb{E}:K} \ar@{>->}[r]^-k &B \ar[r]^-f &M}$ is an internal commutative monoid in $SExt_M$ with multiplication $\omega\colon \E\times \E\to \E$
\[
\xymatrixcolsep{3pc}
\xymatrix{
\E\times \E: {K\times K} \ar@<4ex>[d]_-{\omega_1} \ar@{>->}[r]^-{k\times_M k} &{\mathrm{Eq}(f)} \ar[d]_-{\omega} \ar@{>>}[r]^-{ff_1=ff_2} &M \ar@{=}[d] \\
\;\;\;\;\;\;\E: {K} \ar@{>->}[r]_-k &B \ar@{>>}[r]_-f &{M}
}
\]
and unit $s\colon \mathds{1}\to \E$
\[
\xymatrix{
\mathds{1}: {0} \ar@<1.8ex>[d] \ar@{>->}[r] &{M} \ar[d]_-s \ar@{=}[r] &M \ar@{=}[d] \\
\E: {K} \ar@{>->}[r]_-k &B \ar@{>>}[r]_-f &{M.}
}
\]
Then, as $fs=id_M$ and $s$ preserves the representatives $\{m\}$ of $m$ for $\mathds{1},$ the split extension $\xymatrix{{K} \ar@{>->}[r]^-k &B \ar@<.5ex>@{>>}[r]^-f &M \ar@<.5ex>[l]^-s}$ is a Schreier point on $M$ with Schreier retraction $q\colon B \dashrightarrow K$ defined by $q(b)=q_{f(b),sf(b)}(b),$ for all $b\in B$ (see Remark~\ref{rmk:Shreier exts vs Schreier points}). Moreover, for all $a\in K$ we have $\omega\big(0,k(a)\big)=\omega\big(sfk(a),k(a)\big)=k(a)$ by the unit axiom of Definition~\ref{dfn:int_comm_monoid}. The commutativity of $\omega$ gives $\omega\big(k(a),0\big)=\omega\big(0,k(a)\big)=k(a).$ We deduce that $\omega_1(0,a)=\omega_1(a,0)=a.$ Then, as $\omega_1$ is a monoid homomorphism, by the Eckmann-Hilton argument we conclude that $\omega_1=+,$ the monoid operation on $K,$ which is thus commutative; consequently, $\E$ is an object of $c\text{-}SExt_M.$ If $(x,y)\in\mathrm{Eq}(f),$ $f(x)=f(y)=m,$ we have
\begin{equation*}
\begin{split}
\omega(x,y)&=\omega\big(kq(x)+s(m),y\big)\\
&=\omega\Big(\big(kq(x),0\big)+\big(s(m),y\big)\Big)\\
&=\omega\big(kq(x),0\big)+\omega\big(s(m),y\big)\\
&=kq(x)+\omega\big(sf(y),y\big)\\
&=kq(x)+y\\
&=\mu(x,y);
\end{split}
\end{equation*}
here we used~\eqref{eqn:equality for Schreier point} and the unit axiom of Definition~\ref{dfn:int_comm_monoid}. 

This proves that the above given multiplication $\omega\colon \E\times \E \to \E$ coincides with the deduced one in~\eqref{eqn:mu}.

Eventually, it is clear that when the cancellativity of the kernels is also considered, this equivalence restricts to an equivalence $cc\text{-}SPt_M\cong\mathbf{CMon}(cc\text{-}SExt_M).$
\end{proof}

\begin{remark}
Observe that if $\xymatrix{{K} \ar@{>->}[r]_-k &B \ar@/_1pc/@{-->}[l]_-{q} \ar@<.5ex>@{>>}[r]^-f &M \ar@<.5ex>[l]^-s}$ is a Schreier point on $M$ with $K$ commutative and cancellative, the internal multiplication $\mu(x,y)=kq(x)+y$ defined on the cc-Schreier extension $\xymatrix{{\mathbb{E}:K} \ar@{>->}[r]^-k &B \ar@{>>}[r]^-f &M}$ given in Theorem~\ref{thm:internal_monoids} can be expressed in terms of the connector $p\colon\P{E}\longrightarrow X$ as
\[
\mu(x,y)=p\big(y,kq(y)+y,kq(x)+kq(y)+y\big)
\]
(see~\eqref{eqn: def of p_E}).
\end{remark}

\begin{corollary}
Given the Schreier point~\eqref{eqn:Schreier point overline f}, there results a natural internal commutative monoid structure on the Schreier extension $\overline{\E}: \xymatrix{K \ar@{>->}[r]^-{\overline{\kappa}} &{df} \ar@{>>}[r]^-{\overline{f}} &{M}}$ as in~\eqref{eqn:Schreier point overline f}, whose unit is $\overline{s}\colon \mathds{1}\to \overline{\E}$ and whose multiplication is the morphism $\overline{\mu}\colon \overline{\E}\times \overline{\E} \to \overline{\E},$ where $\overline{\mu}\colon \mathrm{Eq}(\overline{f})\longrightarrow df$ is given by
\begin{equation}
\label{eqn:overline_mu}
\overline{\mu}\big(\gamma(x,k(a)+x),\gamma(y,k(b)+y)\big)=\gamma\big(x,k(a)+k(b)+x\big) (=\gamma\big(y,k(a)+k(b)+y\big)).
\end{equation}
\end{corollary}
\begin{proof}
The result follows immediately from Proposition~\ref{prop:df_schreier}, Theorem~\ref{thm:internal_monoids} and Corollary~\ref{cor:a_equals_b}(2).
\end{proof}

When no ambiguity occurs, we will denote this monoid operation on $df$ simply by $\overline{\mu}\big(\gamma(x,y),\gamma(z,w)\big)=\gamma(x,y)\cdot\gamma(z,w).$

\begin{remark}
It is easy to prove that with respect to this monoid operation $\overline{\mu}$ on $df,$ $\gamma$ satisfies the so-called \emph{Chasles identities} $\gamma(x,x)=1$ and $\gamma(x,y)\cdot\gamma(y,z)=\gamma(x,z)$ (cf.~\cite{bourn-direction}). Indeed, the former is immediate by the definition of the unit $\overline{s}\colon \mathds{1}\to \overline{\E}$ (since $\overline{s}(m)=\gamma(x,x),$ for any $m=f(x)\in M$), and for the latter we have
\begin{equation*}
\resizebox{\textwidth}{!}
{$\begin{split}
\gamma\big(x,k(a)+x\big)\cdot\gamma\big(k(a)+x,k(b)+k(a)+x\big)&=\gamma\big(x,k(a)+x\big)\cdot\gamma\big(k(a)+x,k(a)+k(b)+x\big)\\
&=\gamma\big(x,k(a)+x\big)\cdot \gamma\big(x,k(b)+x\big)\\
&=\gamma\big(x,k(a)+k(b)+x\big),
\end{split}$}
\end{equation*}
using the commutativity of $K,$ and the equalities~\eqref{eqn: def of gamma} and~\eqref{eqn:overline_mu}.
\end{remark}

The construction~\eqref{eqn:direction_D} results in an association of objects
\begin{equation}
\label{eqn:direction_functor}
\begin{array}{rcl}
d\colon cc\text{-}SExt_M & \longrightarrow  & \mathbf{CMon}(cc\text{-}SExt_M), \vspace{5pt}\\
 \mathbb{E} & \longmapsto & d(\mathbb{E})=(\overline{\E},\overline{\mu}\colon \overline{\E}\times \overline{\E} \to \overline{\E},\overline{s}\colon \mathds{1}\to \overline{\E})
\end{array}
\end{equation}
which we claim to be functorial. Observe that for every morphism $\alpha\colon \E\to \E^\prime$
\begin{equation}
\begin{aligned}
\label{eqn:mor}
\xymatrix{
{\mathbb{E}:K} \ar@<1.7ex>[d]_-{\alpha_1} \ar@{>->}[r]^-k &X \ar[d]^-{\alpha} \ar@{->>}[r]^-f &M \ar@{=}[d]\\
{\mathbb{E}^\prime:K^\prime} \ar@{>->}[r]_-{k^\prime} &{X^\prime} \ar@{->>}[r]_-{f^\prime} &{M},
}
\end{aligned}
\end{equation}
we deduce two monoid homomorphisms
\begin{equation*}
\begin{array}{rcl}
{R}(\alpha)\colon \R{E} &\longrightarrow & R_{\mathbb{E}^\prime},  \\
\big(x,k(a)+x\big)&\longmapsto &\big(\alpha(x),\alpha k(a)+\alpha(x)\big)=\big(\alpha(x),k^\prime\alpha_1(a)+\alpha(x)\big)
\end{array}
\end{equation*}
and
\begin{equation*}
\resizebox{\textwidth}{!}
{$\begin{array}{rcl}
{P}(\alpha)\colon \P{E}& \longrightarrow & P_{\mathbb{E}^\prime}, \\
\big(x,k(a)+x,k(b)+k(a)+x\big) &\longmapsto&\big(\alpha(x),\alpha k(a)+\alpha(x),\alpha k(b)+\alpha k(a)+\alpha(x)\big)
=\\ & & =\big(\alpha(x),k^\prime \alpha_1(a)+\alpha(x),k^\prime \alpha_1(b)+k^\prime \alpha_1(a)+\alpha(x)\big)
\end{array}$}
\end{equation*}
making the left-hand side of the following diagram commute
\[
\xymatrixcolsep{3pc}
\xymatrixrowsep{3pc}
\xymatrix{
{\P{E}} \ar[d]_-{P(\alpha)} \ar@<.5ex>[r]^-{\pi_1} \ar@<-.5ex>[r]_-{p_2} &{\R{E}}  \ar[d]^-{R(\alpha)} \ar@{->>}[r]^-{\gamma} &{df} \ar@{.>}[d]^-{d(\alpha)} \\
P_{\mathbb{E}'} \ar@<.5ex>[r]^-{\pi^\prime_1} \ar@<-.5ex>[r]_-{p^\prime_2} & R_{\mathbb{E}'} \ar@{->>}[r]_-{\gamma^\prime} &{df^\prime}.
}
\]
By the universal property of the coequaliser $\gamma,$ this induces a unique morphism $d(\alpha):df\rightarrow df^\prime$ in $\mathbf{Mon}$ such that $d(\alpha)\gamma=\gamma^\prime {R}(\alpha).$ In terms of elements, we have $d(\alpha)\big(\gamma(x,k(a)+x)\big)=\gamma^\prime\big(\alpha(x),k^\prime\alpha_1(a)+\alpha(x)\big).$

Observe now that with this definition of $d(\alpha)$ the diagram
\[
\xymatrix{
{K} \ar[d]_-{\alpha_1} \ar@{>->}[r]^-{\overline{\kappa}} &{df} \ar[d]_-{d(\alpha)} \ar@<.5ex>@{>>}[r]^-{\overline{f}} &M \ar@{=}[d] \ar@<.5ex>[l]^-{\overline{s}} \\
{K^\prime} \ar@{>->}[r]_-{\overline{\kappa}^\prime} &{df^\prime} \ar@<.5ex>@{>>}[r]^-{\overline{f^\prime}} &{M} \ar@<.5ex>[l]^-{\overline{s}^\prime}
}
\]
is commutative, so that $d(\alpha)$ is a morphism in $cc\text{-}SPt_M,$ and consequently, by Theorem~\ref{thm:internal_monoids}, it determines a morphism $d(\alpha)\colon d(\mathbb{E})\rightarrow d(\mathbb{E^\prime})$ in $\mathbf{CMon}(cc\text{-}SExt_M).$ It is clear that this definition of $d$ on the morphisms $\alpha\colon\mathbb{E}\rightarrow\mathbb{E}^\prime$ preserves compositions and identities, allowing us to conclude that~\eqref{eqn:direction_functor} is indeed a functor, which we call the \emph{direction functor} for cc-Schreier extensions. We are entitled to call it so, because we know by Remark~\ref{rmk:group_slice} that when $\mathbb{E}$ is an extension of groups, the connector $p=p_{\mathbb{E}}$ is the unique autonomous Mal'tsev operation associated with $f\in \mathbf{Gp}/M,$ and in this case our functor $d$ coincides, by construction, with the classical direction functor \cite{bourn-direction} applied to the slice category $\mathbf{Gp}/M.$

As it happens for the aforementioned case of groups, it follows by Theorem~\ref{thm:internal_monoids} that the functor~\eqref{eqn:direction_functor} admits an explicit description in terms of monoid \emph{actions}. Recall by Proposition~\ref{prop:actions} that any Schreier point $\xymatrix{{K} \ar@{>->}[r]_-k &B \ar@/_1pc/@{-->}[l]_-{q} \ar@<.5ex>@{>>}[r]^-f &M \ar@<.5ex>[l]^-s}$ determines an action $\sigma\colon M\to \mathrm{End}(K)$ of $M$ on $K$ defined by $\sigma(m)(a)=m\bullet a=q\big(s(m)+k(a)\big)$ (see~\eqref{eqn:Schreier point action sigma}), which allows to describe $B$ as the semidirect product $B\cong K\rtimes_{\sigma}M$ thanks to the Split Short Five Lemma (valid for Schreier points, see \cite[Corollary 2.3.8]{schreier_book}). More precisely, we have an isomorphism in $SPt_M$
\[
\xymatrix@C=30pt{
{K} \ar@{=}[d] \ar@{>->}[r]_-{\la id_K,0\ra} & K\rtimes_{\sigma}M \ar@/_1pc/@{-->}[l]_-{\pi_K} \ar[d]_-{\varphi} \ar@<.5ex>@{>>}[r]^-{\pi_M} &M \ar@{=}[d] \ar@<.5ex>[l]^-{\la 0, id_M\ra} \\
{K} \ar@{>->}[r]_-{k} &{B} \ar@/_1pc/@{-->}[l]_-{q} \ar@<.5ex>@{>>}[r]^-{f} &{M}, \ar@<.5ex>[l]^-{s}
}
\]
where $ K\rtimes_{\sigma}M=(K\times M, +)$ (the monoid operation is given by $(a,m)+(a',m')=(a+m\bullet a',m\cdot m'),$ for $a,a'\in K$ and $m,m'\in M$), and the morphism $\varphi\colon K\rtimes_{\sigma}M\rightarrow B,$ $\varphi(a,m)=k(a)+s(m),$ is an isomorphism in $\mathbf{Mon}.$

When $K$ is commutative, one says that $(K,\sigma)$ is an $M$\emph{-semimodule}. We denote by $\mathcal{S}mod_M$ the category of $M$-semimodules and action preserving homomorphisms: $h\colon(K,\sigma) \to (K',\sigma')$ is a morphism in $\mathcal{S}mod_M$ if $h\colon K\to K'$ is a monoid homomorphism such that the following diagram commutes
\[
\xymatrix{M\times K \ar@{-->}[r] \ar[d]_-{id_M\times h} & K \ar[d]^-h \\ M\times K' \ar@{-->}[r] & K',}
\]
i.e., $h(\sigma(m)(a))=\sigma'(m)(h(a))$ for all $a\in K$ and $m\in M.$ 

It is easy to check that $c\text{-}SPt_M\cong \mathcal{S}mod_M.$ The equivalence is obtained by associating with every Schreier point $\xymatrix{{K} \ar@{>->}[r]_-k &B \ar@/_1pc/@{-->}[l]_-{q} \ar@<.5ex>@{>>}[r]^-f &M \ar@<.5ex>[l]^-s}$ with commutative $K$ the $M$-semimodule $(K,\sigma),$ where $\sigma$ is defined in~\eqref{eqn:Schreier point action sigma}, and, conversely, with every $M$-semimodule $(K,\sigma)$ the Schreier point $\xymatrixcolsep{2.5pc}\xymatrix{{K} \ar@{>->}[r]_-{\la id_K,0 \ra} &{K\rtimes_{\sigma} M} \ar@/_1pc/@{-->}[l]_-{\pi_K} \ar@<.5ex>@{>>}[r]^-{\pi_M} &{M} \ar@<.5ex>[l]^-{\la 0,id_M\ra}}$ (with $\pi_K$ and $\pi_M$ projections).

In our case, we have:
\begin{proposition}
Given a cc-Schreier extension $\E$ as above, the monoid action of $M$ on $K$ induced by the Schreier point~\eqref{eqn:Schreier point overline f} coincides with the monoid action~\eqref{eqn:action} induced by $\mathbb{E}.$
\end{proposition}
\begin{proof}
By Proposition~\ref{prop:df_schreier}, the Schreier retraction of~\eqref{eqn:Schreier point overline f} is given by $\overline{q}(\gamma(x,k(a)+x))=a.$ Hence, for all $m=f(x)\in M$ and $a\in K,$ the associated $M$-action on $K$ is $m\bullet a=\overline{q}\big(\overline{s}(m)+\overline{\kappa}(a)\big)=\overline{q}\gamma\big(x,x+k(a)\big)=\overline{q}\gamma\big(x,k(m\ast a)+x\big)=m\ast a$ (using~\eqref{eqn:Patrick}), the same as in~\eqref{eqn:action}.
\end{proof}

If we denote by $\eta\colon M\to \mathrm{End}(K)$ the monoid homomorphism which gives the action~\eqref{eqn:action} induced by $\mathbb{E}$ (see~\eqref{eqn:Schreier ext action eta}), it follows that $df\cong K\rtimes_{\eta}M$ via $\varphi\colon K\rtimes_{\eta}M\longrightarrow df,$ $(a,m)\mapsto \overline{\kappa}(a)+\overline{s}(m)=\gamma\big(x,k(a)+x\big),$ where $x\in X$ is any element such that $f(x)=m.$ We conclude that the direction functor defined in~\eqref{eqn:direction_functor} can be interpreted as the functor $d\colon cc\text{-}SExt_M\longrightarrow \mathcal{S}mod_M$ associating with each cc-Schreier extension $\mathbb{E}$ the $M$-semimodule $(K,\eta),$ and with each morphism~\eqref{eqn:mor} of cc-Schreier extensions, the morphism of $M$-semimodules $\alpha_1\colon(K,\eta)\rightarrow(K^\prime,\eta^\prime).$ Indeed, this morphism is action preserving, meaning that $\alpha_1(m\ast a)=m\ast' \alpha_1(a),$ for all $a\in K$ and $m\in M,$ thanks to the fact that in~\eqref{eqn:mor} $\alpha$ preserves the representatives.

\section{Properties of the direction functor}
Here we collect the main properties of the direction functor $d$ introduced in the previous section. We shall freely use both descriptions of $d,$ either as the functor
\[
\begin{array}{rcl}
d\colon cc\text{-}SExt_M & \longrightarrow & \mathbf{CMon}(cc\text{-}SExt_M)\cong cc\text{-}SPt_M \\
	\xymatrix{{\mathbb{E}:K} \ar@{>->}[r]^-k &X \ar@{->>}[r]^-f &{M}} & \longmapsto & \xymatrix{{K} \ar@{>->}[r]_-{\overline{\kappa}} & {df} \ar@/_1pc/@{-->}[l]_-{\overline{q}} \ar@<.5ex>@{>>}[r]^-{\overline{f}} &{M,} \ar@<.5ex>[l]^-{\overline{s}}}
\end{array}
\]
attaching the Schreier point~\eqref{eqn:Schreier point overline f} to any cc-Schreier extension $\mathbb{E}$ or as the action functor
\[
\begin{array}{rcl}
d\colon cc\text{-}SExt_M & \longrightarrow & \mathcal{S}mod_M \\
\xymatrix{{\mathbb{E}:K} \ar@{>->}[r]^-k &X \ar@{->>}[r]^-f &{M}} & \longmapsto & (K,\eta\colon M\to \mathrm{End}(K))
\end{array}
\]
associating with $\mathbb{E}$ the $M$-semimodule structure on $K$ given by the monoid action~\eqref{eqn:action}. \\

Observe that the functor $d\colon cc\text{-}SExt_M \longrightarrow \mathcal{S}mod_M$ can be defined, more generally, on the full subcategory $smod\text{-}SExt_M \subseteq c\text{-}SExt_M$ of all Schreier extensions inducing an $M$-semimodule structure, namely Schreier extensions~\eqref{eqn:ses} with commutative $K$ such that the map~\eqref{eqn:action} is an action (indeed, we know from Section~\ref{sec:action} that the cancellativity of the kernel $K$ is only a sufficient condition for this to happen). In the remainder of the section, we shall point out the good properties of both $d$ and this broadened functor $D\colon smod\text{-}SExt_M\longrightarrow \mathcal{S}mod_M$ (such that $d=D$ on $cc\text{-}SExt_M$).

\begin{proposition}
\label{prop:d_conservative}
The functors $d$ and $D$ are conservative.
\end{proposition}

\begin{proof}
Consider a morphism $\alpha\colon \E \to \E^\prime$ as in~\eqref{eqn:mor} in $cc\text{-}SExt_M$ such that
\[
\xymatrix{
{K} \ar[d]_-{\alpha_1} \ar@{>->}[r]^-{\overline{\kappa}} &{df} \ar[d]_-{d(\alpha)} \ar@<.5ex>@{>>}[r]^-{\overline{f}} &M \ar@{=}[d] \ar@<.5ex>[l]^-{\overline{s}} \\
{K^\prime} \ar@{>->}[r]_-{\overline{\kappa^\prime}} &{df^\prime} \ar@<.5ex>@{>>}[r]^-{\overline{f^\prime}} &{M} \ar@<.5ex>[l]^-{\overline{s^\prime}}
}
\]
is an isomorphism in $cc\text{-}SPt_M.$ Then, in particular, $\alpha_1$ is an isomorphism of monoids, and it follows from the Short Five Lemma (Proposition~\ref{prop:short_five_lemma_schreier}) applied to~\eqref{eqn:mor} that $\alpha$ is also an isomorphism. Observe that if $u_m^\prime=\alpha(u)\in X^\prime$ is a representative, then $u\in X$ is also a representative, so that $\alpha^{-1}$ preserves the representatives. This shows that $d$ is conservative; the same proof works for $D,$ as well.
\end{proof}

In fact, the conservativity of $d$ and $D$ can be seen as a consequence of the following result:
\begin{proposition}
\label{prop:d_monos_reg_epis}
The functors $d$ and $D$ preserve and reflect monomorphisms and regular epimorphisms.
\end{proposition}
The proof is just an application of the Short Five Lemma in $c\text{-}SExt_M$ and the Split Short Five Lemma in $c\text{-}SPt_M,$ once that the nature of monomorphisms and regular epimorphisms in both categories is established. The full detailed proof is given in Appendix~\ref{appendix:monos}.

\begin{proposition}
\label{prop:d_products}
The direction functors $d$ and $D$ preserve finite products.
\end{proposition}

\begin{proof}
Recall that the trivial extension
\[
\xymatrix{{\mathds{1}:0} \ar@{>->}[r] &M \ar@{=}[r] &M}
\]
is a terminal object in $cc\text{-}SExt_M,$ and observe that given two cc-Schreier extensions
\[
\xymatrix{{\mathbb{E}:K} \ar@{>->}[r]^-k &X \ar@{>>}[r]^f &{M,}}
\]
\[
\xymatrix{{\mathbb{E^\prime}:K^\prime} \ar@{>->}[r]^-{k^\prime} &{X^\prime} \ar@{>>}[r]^{f^\prime} &{M,}}
\]
their product in $cc\text{-}SExt_M$ is given by
\begin{equation}
\label{eqn:product}
\xymatrixcolsep{3.5pc}
\xymatrix{ {\mathbb{E}\times\mathbb{E^\prime}:K\times K^\prime} \ar@{>->}[r]^-{k\times_M k^\prime} &{X\times_M X^\prime} \ar@{>>}[r]^-{f^\prime p_2=fp_1} &{M,}}
\end{equation}
with the obvious projections on $\mathbb{E}$ and $\mathbb{E^\prime},$ where
\[
\xymatrix{
{X\times_M X^\prime} \pullback \ar[r]^-{p_2} \ar[d]_-{p_1} &{X^\prime} \ar[d]^-{f^\prime} \\
X \ar[r]_-{f} &M
}
\]
is a pullback. (By realising $X\times_M X^\prime\cong\big\{(x,x^\prime)\in X\times X^\prime:f(x)=f^\prime(x^\prime)\big\},$ \eqref{eqn:product} is a Schreier extension with representatives $(u_m,u^\prime_m),$ where $u_m$ is a representative of $m$ for $\E$ and $u^\prime_m$ is a representative of $m$ for $\E^\prime.$)

Now, consider the direction functor in its action guise $d\colon cc\text{-}SExt_M\longrightarrow \mathcal{S}mod_M,$ so that $d(\mathbb{E})=(K,\eta),$ where $\eta$ is the induced monoid homomorphism action~\eqref{eqn:Schreier ext action eta}; let $d(\E^\prime)=(K^\prime,\eta^\prime).$ Then $d(\mathds{1})$ is the trivial $M$-semimodule $(0,!_M),$ which is clearly a terminal object in $\mathcal{S}mod_M.$ Suppose next that $d(\E\times \E^\prime)=(K\times K', \nu
),$ where $\nu\colon M\to \mathrm{End}(K\times K^\prime)$ is the monoid homomorphism induced by the product Schreier extension~\eqref{eqn:product}: thus, for any $m\in M,$
\[
\nu(m)\colon K\times K^\prime \longrightarrow K\times K^\prime, \ (a,a^\prime)\mapsto (b,b^\prime),
\]
with
\begin{equation*}
\begin{split}
&(u_m,u^\prime_m)+k\times_M k^\prime(a,a^\prime)=k\times_M k^\prime(b,b^\prime)+(u_m,u^\prime_m)\\
\Leftrightarrow \;\;& (u_m,u^\prime_m)+\big(k(a),k^\prime(a^\prime)\big)=\big(k(b),k^\prime(b^\prime)\big)+(u_m,u^\prime_m).
\end{split}
\end{equation*}
It follows that $u_m+k(a)=k(b)+u_m$ and $u^\prime_m+k^\prime(a^\prime)=k^\prime(b^\prime)+u^\prime_m,$ i.e. $b=\eta(m)(a)$ and $b^\prime=\eta^\prime(m)(a^\prime).$ We conclude that $\nu(m)(a,a^\prime)=(\eta(m)(a),\eta^\prime(m)(a^\prime)),$ which means that $\nu$ is the product action $\eta\times\eta^\prime$ given, more precisely, by the monoid homomorphism
\[
\xymatrix@C=40pt{M \ar[r]^-{\la id_M,id_M\ra} & M\times M \ar[r]^-{\eta\times \eta^\prime} & \mathrm{End}(K)\times \mathrm{End}(K^\prime) \ar@{>->}[r] & \mathrm{End}(K\times K^\prime).}
\]
This proves that $d(\mathbb{E}\times\mathbb{E^\prime})\cong d(\mathbb{E})\times d(\mathbb{E^\prime}).$
Observe that cancellativity is used nowhere in the proof, so the same holds for $D.$
\end{proof}

We come to the main property of $d$ and $D,$ namely the fact that they are cofibrations. Recall the general definition of the latter:

\begin{definition}
\label{def:cofibration}
Let ${F}\colon\mathcal{U}\rightarrow\mathcal{V}$ be a functor.
\begin{enumerate}
\item We say that a morphism $f\colon Y\rightarrow X$ in $\mathcal{U}$ is \emph{cocartesian over a morphism $\alpha\colon J\rightarrow I$ in $\mathcal{V}$} if ${F}f=\alpha$ and, for every morphism $g\colon Y\rightarrow Z$ in $\mathcal{U}$ such that ${F}g=\beta\alpha$ in $\mathcal{V}$ for some $\beta\colon I\rightarrow {F}Z$
\begin{equation*}
\centering
\xymatrix{
&Y \ar[r]^f \ar[d]_g &X\\
&Z  &\
}
\xymatrix{
&\ \ar@{--}[d] \\
&\
}
\xymatrix{
&J  \ar[d]_{{F}g}  \ar[r]^{\alpha={F}f} &I  \ar[dl]^{\exists \, \beta}\\
&{{F}Z,} &\
}
\end{equation*}
one has $g=hf$ for a unique $h\colon X\rightarrow Z$ in $\mathcal{U}$ such that ${F}h=\beta$
\begin{equation*}
\centering
\xymatrix{
&Y \ar[r]^f \ar[d]_g &X \ar@{.>}[dl]^{\exists! \, h}\\
&Z  &\
}
\xymatrix{
&\ \ar@{--}[d] \\
&\
}
\xymatrix{
&J  \ar[d]_{{F}g}  \ar[r]^{\alpha={F}f} &I  \ar[dl]^{\beta={F}h}\\
&{{F}Z.} &\
}
\end{equation*}
A morphism $f$ in $\mathcal{U}$ is \emph{cocartesian} (for ${F}$) if it is cocartesian over $\alpha={F}f.$
\item ${F}\colon\mathcal{U}\rightarrow\mathcal{V}$ is a \emph{cofibration} if for every morphism $\alpha\colon J\rightarrow I$ in $\mathcal{V}$ and for every object $Y$ in the fibre ${F}^{-1}(J)$ there exists a cocartesian morphism in $\mathcal{U}$ over $\alpha$ whose domain is $Y.$
\end{enumerate}
\end{definition}
(This notion, together with its dual notion of \emph{fibration}, goes back to A. Grothendieck \cite{bourbaki}. See also \cite{borceux-bourn}, Appendix $A.7.$) 

Fix a Schreier extension $\xymatrix{{\mathbb{E}:K} \ar@{>->}[r]^-k &X \ar@{>>}[r]^f &{M}}$ in $smod\text{-}SExt_M,$ inducing the action $\eta$ as in~\eqref{eqn:Schreier ext action eta}, and consider a morphism of $M$-semimodules $\alpha_1\colon(K,\eta)=D(\mathbb{E})\longrightarrow(K^\prime,\eta^\prime)$: thus, $\alpha_1$ is a monoid homomorphism satisfying $\alpha_1\big(\eta(m)(a)\big)=\eta^\prime(m)\big(\alpha_1(a)\big)$ for every $m\in M$ and $a\in K.$

We shall proceed by constructing the so-called \emph{pushforward} of $\mathbb{E}$ along $\alpha_1,$ which will give a Schreier extension $\E^\prime$ and a cocartesian morphism $\alpha\colon \E \to \E^\prime$ in $smod\text{-}SExt_M$ above $\alpha_1\colon D(\E)\to (K^\prime, \eta^\prime)$ (see~\eqref{eqn:cocartesian_mor}).

First, observe that by composing the action $\eta^\prime\colon M\longrightarrow \mathrm{End}(K^\prime)$ with $f$ we get an action $\psi=\eta^\prime f\colon X\longrightarrow \mathrm{End}(K^\prime)$ of $X$ on $K^\prime$ such that
\begin{equation}
\label{eqn:action varphi}
	\psi(x)(a')=\eta'(f(x))(a'),\text{ for all }x\in X\text{ and }a'\in K'.
\end{equation}
 We can then form the semidirect product $K^\prime\rtimes_{\psi} X$ and consider the square
\[
\xymatrixcolsep{3pc}
\xymatrixrowsep{3pc}
\xymatrix{
K \ar[r]^-k \ar[d]_-{\alpha_1} &{X} \ar[d]^-{\la 0,id_X\ra}  \\
{K^\prime} \ar[r]_-{\la id_{K^\prime},0 \ra} &{K^\prime\rtimes_{\psi} X,}
}
\]
which is \emph{not} commutative, in general. Define $\big(X^\prime,\xymatrix{{r\colon K^\prime\rtimes_{\psi} X} \ar@{->>}[r] &{X^\prime}})$ to be a coequaliser of $\langle id_{K^\prime},0 \rangle \alpha_1$ and $\langle 0,id_X \rangle k$ in $\mathbf{Mon},$ so that $X^\prime\cong {(K^\prime\rtimes_{\psi} X)}/{_\sim}$ where $\sim$ is the internal equivalence relation on $K^\prime\rtimes_{\psi} X$ generated by $\big(\alpha_1(a),0\big)\sim\big(0,k(a)\big)$ for all $a\in K.$ We have:

\begin{lemma}[{Cf.~\cite{P-II}}]
\label{lm:rho}
The above relation $\sim$ coincides with the relation $\rho$ on $K^\prime\rtimes_{\psi} X$ defined by $\big(a^\prime,x\big)\rho\big(b^\prime,y\big)$ if and only if $(x,y)\in\mathrm{Eq}(f)$ and $a^\prime+\alpha_1(q(x))=b^\prime+\alpha_1(q(y)),$ where $x=kq(x)+u_m$ and $y=kq(y)+u_m$ with respect to some representative $u_m$ of $m=f(x)=f(y)$ (using~\eqref{eqn:x=kq(x)+u}).
\end{lemma}

\begin{proof}
First observe that $\rho$ is well defined, meaning that it does not depend upon the choice of $u_m,$ because any two representatives differ by an invertible element of $K$ (Lemma~\ref{lemma:rep_invertible}) and the image of an invertible element under the monoid homomorphism $\alpha_1$ is still invertible in $K^\prime.$ Moreover, $\rho$ is clearly an equivalence relation on $K^\prime\rtimes_\psi X,$ and it is shown in \cite[Proposition~4.7]{P-II} that $\rho$ is an internal relation in $\mathbf{Mon}$ (using~\eqref{eqn:Patrick}). Observe also that for every $a\in K$ one has $\big(\alpha_1(a),0\big)\rho\big(0,k(a)\big),$ by choosing $0\in X$ as a representative of $1\in M,$ so that $\sim$ is contained in $\rho.$ Conversely, suppose that $\big(a^\prime,x\big)\rho\big(b^\prime,y\big)$ and write $x=kq(x)+u_m,$ $y=kq(y)+u_m,$ where $m=f(x)=f(y).$ Then $\big(a^\prime+\alpha_1(q(x)),u_m\big)\sim\big(b^\prime+\alpha_1(q(y)),u_m\big)$ because $\sim$ is reflexive, and
\[
\big(a^\prime+\alpha_1(q(x)),u_m\big)=\big(\alpha_1(q(x)),0\big)+\big(a^\prime,u_m\big)\sim\big(0,kq(x)\big)+\big(a^\prime,u_m\big)=\big(a^\prime,x\big)
\]
(using the commutativity of $K^\prime,$ the generators of $\sim$ and the definition of the monoid operation on $K^\prime\rtimes_{\psi} X$). Similarly, $\big(b^\prime+\alpha_1(q(y)),u_m\big)\sim \big(b^\prime,y\big),$ which allows us to conclude that $\big(a^\prime,x\big)\sim\big(b^\prime,y\big).$
\end{proof}

Then, we can compute the coequaliser $(X^\prime,r)$ as $K^\prime\rtimes_{\psi} X\xlongrightarrow{r} X^\prime={(K^\prime\rtimes_{\psi} X)}/{\rho},$ $(a^\prime,x)\mapsto[(a^\prime,x)]_\rho,$ and we have a commutative square
\[
\xymatrixcolsep{3.5pc}
\xymatrixrowsep{2.5pc}
\xymatrix{
K \ar@{>->}[r]^-k \ar[d]_-{\alpha_1} &{X} \ar[d]^-{\alpha=r\la 0,id_X\ra}  \\
{K^\prime} \ar[r]_-{k'=r\la id_{K^\prime},0\ra} &{X^\prime.}
}
\]
Observe that the morphism $k'=r\la id_{K^\prime},0\ra\colon K^\prime\rightarrow X^\prime$ is a monomorphism, because it follows from the definition of $\rho$ that $(a^\prime,0)\rho(b^\prime,0)$ if and only if $a^\prime=b^\prime.$ Moreover, if $\pi_X$ denotes the projection morphism $K^\prime\rtimes_{\psi} X\rightarrow X,(a^\prime,x)\mapsto x,$ the equality $f\pi_X\la 0,id_X\ra k=fk=0=f\pi_X\la id_{K^\prime},0 \ra\alpha_1$ and the universal property of the coequaliser guarantee that a unique $f^\prime$ exists as in the diagram
\[
\xymatrixcolsep{3.5pc}
\xymatrix{
{K} \ar@<.7ex>[r]^-{\la 0,id_X \ra k} \ar@<-.7ex>[r]_-{\la id_{K^\prime},0\ra \alpha_1} &{K^\prime\rtimes_{\psi} X} \ar@{>>}[d]_-{f\pi_X} \ar@{->>}[r]^-r &{X^\prime} \ar@{.>}[dl]^-{f^\prime} \\
&{M} &\
}
\]
satisfying $f\pi_X=f^\prime r.$ Explicitly, $f^\prime$ is defined by $f^\prime\big([(a^\prime,x)]_\rho\big)=f(x),$ and observe that as $f$ and $\pi_X$ are regular epimorphisms in $\mathbf{Mon},$ so is $f^\prime.$ Consequently, we have $Ker(f^\prime)=\{[(a^\prime,x)]_\rho\in X^\prime:x\in Ker(f)\}=k^\prime(K^\prime),$ where $k^\prime=r \langle id_{K^\prime},0 \rangle,$ because $\big[\big(a^\prime,k(a)\big)\big]_\rho=\big[\big(a^\prime+\alpha_1(a),0\big)\big]_\rho=k^\prime\big(a^\prime+\alpha_1(a)\big)$ for every $a\in K$ and $a^\prime\in K^\prime,$ and as we already know that $k^\prime$ is a monomorphism we conclude that $(K^\prime,k^\prime)$ is a kernel of $f^\prime.$

We obtain a monoid extension $\xymatrix{{\mathbb{E^\prime}:K^\prime} \ar@{>->}[r]^-{k^\prime} &{X^\prime} \ar@{>>}[r]^{f^\prime} &{M}}$ (which is the pushforward of $\E$ along $\alpha_1$) and we see that the diagram
\begin{equation}
\begin{aligned}
\label{eqn:cocartesian_mor}
\xymatrix{
{\mathbb{E}:K} \ar@<1.7ex>[d]_-{\alpha_1} \ar@{>->}[r]^-k &X \ar[d]^-{\alpha} \ar@{->>}[r]^-f &M \ar@{=}[d]\\
{\mathbb{E}^\prime:K^\prime} \ar@{>->}[r]_-{k^\prime} &{X^\prime} \ar@{->>}[r]_-{f^\prime} &{M}
}
\end{aligned}
\end{equation}
commutes (because $f'\alpha=f^\prime r\la 0,id_X\ra=f\pi_X\la 0,id_X\ra=f$).

\begin{lemma}
The sequence $\mathbb{E^\prime}$ and the morphism~\eqref{eqn:cocartesian_mor} belong to $smod\text{-}SExt_M$ and, moreover, $D(\mathbb{E^\prime})=(K^\prime,\eta^\prime).$
\end{lemma}
\begin{proof}
Every $[(a^\prime,x)]_\rho\in X^\prime,$ where $x=kq(x)+u_{f(x)}$ by~\eqref{eqn:x=kq(x)+u}, can be written as $[(a^\prime,x)]_\rho=[(a^\prime+\alpha_1(q(x)),u_{f(x)})]_\rho=[(a^\prime+\alpha_1(q(x)),0)]_\rho+[(0,u_{f(x)})]_\rho=k'(a'+\alpha_1(q(x)))+[(0,u_{f(x)})]_\rho,$ and it follows from the definition of $\rho$ that if $[(a^\prime,x)]_\rho=k'(b')+[(0,u_{f(x)})]_\rho=[(b^\prime,u_{f(x)})]_\rho$ for some other $b^\prime\in K^\prime,$ then $a^\prime+\alpha_1(q(x))=b^\prime.$ Thus $\mathbb{E^\prime}$ is a Schreier extension with representatives $[(0,u_m)]_\rho=\alpha(u_m);$ consequently $\alpha\colon\E\to \E'$ is a morphism of Schreier extensions (see Definition~\ref{dfn:morphism of Schreier exts}). The only thing that we are left to prove is that the map $M\times K^\prime\dashrightarrow K^\prime$ associated with $\mathbb{E^\prime}$ as in~\eqref{eqn:action} corresponds to the one induced from $\eta^\prime.$ But this is true, because $M\times K^\prime\dashrightarrow K^\prime$ is defined by $(m,a^\prime)\mapsto a^{\prime\prime},$
where $a''$ is the unique element of $K'$ such that $[(0,u_m)]_\rho+k'(a')=k'(a'')+[(0,u_m)]_\rho.$ This gives
\begin{equation*}
\begin{split}
	& [(0,u_m)]_\rho+[(a',0)]_\rho = [(a'',0)]_\rho +[(0,u_m)]_\rho \\
	& \Leftrightarrow [(\psi(u_m)(a'),u_m)]_\rho = [(a'',u_m)]_\rho \\
	& \Leftrightarrow \psi(u_m)(a')=a'' \\
	& \Leftrightarrow \eta'(m)(a')=a'',
\end{split}
\end{equation*}
using the definition of $\rho$ in Lemma~\ref{lm:rho} and~\eqref{eqn:action varphi}.
\end{proof}

We claim that~\eqref{eqn:cocartesian_mor} is a cocartesian morphism over $\alpha_1,$ in the sense of Definition~\ref{def:cofibration}, for the functor $D\colon smod\text{-}SExt_M \longrightarrow \mathcal{S}mod_M.$ Indeed, suppose that a morphism
\[
\xymatrix{
{\mathbb{E}:K} \ar@<1.7ex>[d]_-{\lambda_1} \ar@{>->}[r]^-k &X \ar[d]^-{\lambda} \ar@{->>}[r]^-f &M \ar@{=}[d]\\
{\mathbb{F}:L} \ar@{>->}[r]_-{l} &{Y} \ar@{->>}[r]_-{g} &{M}
}
\]
is given in $smod\text{-}SExt_M,$ such that $\lambda_1=\beta_1\alpha_1$ for some $\beta_1\colon(K^\prime,\eta^\prime)\rightarrow(L,\nu)$ in $\mathcal{S}mod_M$
\[
\xymatrix{
{(K,\eta)} \ar[r]^-{\alpha_1} \ar[d]_-{\lambda_1} &{(K^\prime,\eta^\prime)} \ar[ld]^-{\beta_1} \\
{(L,\nu),} &\
}
\]
where $(L,\nu)=D(\mathbb{F}).$

We define $h\colon K^\prime\rtimes_\psi X\longrightarrow Y,$ $(a^\prime,x)\mapsto l\beta_1(a^\prime)+\lambda(x),$ which is a monoid homomorphism because
\begin{equation*}
\begin{split}
h\big((a^\prime,x) + (b^\prime,y)\big)&=h \big(a^\prime+\psi(x)(b^\prime),x+y\big)\\
&=h \big(a^\prime+\eta^\prime(f(x))(b^\prime),x+y\big)\\
&=l\beta_1\big(a^\prime+\eta^\prime(f(x))(b^\prime)\big)+\lambda(x+y)\\
&=l\beta_1(a')+l\beta_1\big(\eta'(f(x))(b')\big)+\lambda(x)+\lambda(y)\\
&=l\beta_1(a')+l\big(\nu(f(x))(\beta_1(b'))\big)+\lambda(x)+\lambda(y)\\
&=l\beta_1(a')+l\big(\nu(g\lambda(x))(\beta_1(b'))\big)+\lambda(x)+\lambda(y)\\
&=l\beta_1(a^\prime)+\lambda(x)+l\beta_1(b^\prime)+\lambda(y)\\
&=h(a^\prime,x)+h(b^\prime,y)
\end{split}
\end{equation*}
(using~\eqref{eqn:action varphi}, the fact that $\beta_1$ is a morphism in $\mathcal{S}mod_M$ and~\eqref{eqn:Patrick}). We have then $h\la 0,id_X\ra k=h \la id_{K^\prime},0 \ra\alpha_1,$ because $\lambda k=l\lambda_1=l\beta_1\alpha_1.$ By the universal property of the coequaliser $r,$ a unique $\beta\colon X^\prime\rightarrow Y$ is induced
\begin{equation}
\begin{aligned}
\label{eqn:unique_beta}
\xymatrixcolsep{3.5pc}
\xymatrix{
{K} \ar@<.7ex>[r]^-{\la 0,id_X\ra k} \ar@<-.7ex>[r]_-{\la id_{K^\prime},0\ra\alpha_1} &{K^\prime\rtimes_{\psi} X} \ar[d]_-{h} \ar@{->>}[r]^-r &{X^\prime} \ar@{.>}[dl]^-{\beta} \\
&{Y} &\
}
\end{aligned}
\end{equation}
satisfying $h=\beta r,$ explicitly (well) defined by $\beta\big([(a^\prime,x)]_\rho\big)=h(a',x)=l\beta_1(a^\prime)+\lambda(x).$\\The diagram
\begin{equation}
\begin{aligned}
\label{eqn:cocartesian_lifting}
\xymatrix{
{\mathbb{E^\prime}:K^\prime} \ar@<1.7ex>[d]_-{\beta_1} \ar@{>->}[r]^-{k^\prime} &{X^\prime} \ar[d]^-{\beta} \ar@{->>}[r]^-{f^\prime} &M \ar@{=}[d]\\
{\mathbb{F}:L} \ar@{>->}[r]_-{l} &{Y} \ar@{->>}[r]_-{g} &{M}
}
\end{aligned}
\end{equation}
is also commutative, because $\beta k^\prime(a^\prime)=\beta\big([(a^\prime,0)]_\rho\big)=l\beta_1(a^\prime)$ and $g\beta\big([(a^\prime,x)]_\rho\big)=g\big(l\beta_1(a^\prime)+\lambda(x)\big)=0+g\lambda(x)=f(x)=f^\prime\big([(a^\prime,x)]_\rho\big)$ for every $a^\prime\in K^\prime$ and $x\in X.$ Moreover, $\beta$ preserves the representatives because, by assumption, $\lambda$ does and $\beta\big([(0,u_m)]_\rho\big)=\lambda(u_m)$: thus~\eqref{eqn:cocartesian_lifting} is a morphism in $smod\text{-}SExt_M,$ and as for every $x\in X$ the equality $\beta\alpha(x)=\beta\big([(0,x)]_\rho\big)=\lambda(x)$ holds (so that $\beta\alpha=\lambda$), the triangle
\begin{equation}
\begin{aligned}
\label{eqn:beta_D}
\xymatrixcolsep{2.5pc}
\xymatrix{
{\mathbb{E}} \ar[r]^-{(\alpha_1,\alpha)} \ar[d]_-{(\lambda_1,\lambda)} &{\mathbb{E^\prime}} \ar[ld]^-{(\beta_1,\beta)} \\
{\mathbb{F}} &\
}
\end{aligned}
\end{equation}
commutes in $smod\text{-}SExt_M.$

To prove the uniqueness of $\beta$ in~\eqref{eqn:beta_D}, suppose that for some other morphism
\[
\xymatrix{
{\mathbb{E^\prime}:K^\prime} \ar@<1.7ex>[d]_-{\beta_1} \ar@{>->}[r]^-{k^\prime} &{X^\prime} \ar[d]^-{\overline{\beta}} \ar@{->>}[r]^-{f^\prime} &M \ar@{=}[d]\\
{\mathbb{F}:L} \ar@{>->}[r]_-{l} &{Y} \ar@{->>}[r]_-{g} &{M}
}
\]
in $smod\text{-}SExt_M$ we have $\overline{\beta}\alpha=\lambda.$ Then for all $(a^\prime,x)\in K^\prime\rtimes_\psi X$
\begin{equation*}
\begin{split}
\overline{\beta} r(a^\prime,x)&=\overline{\beta}\big([(a^\prime,x)]_\rho\big)\\
&=\overline{\beta}\big([(a^\prime,0)]_\rho+[(0,x)]_\rho\big)\\
&=\overline{\beta}\big([(a^\prime,0)]_\rho\big)+\overline{\beta}\big([(0,x)]_\rho\big)\\
&=\overline{\beta} k^\prime(a^\prime)+\overline{\beta}\alpha(x)\\
&=l\beta_1(a^\prime)+\lambda(x)\\
&=h(a',x),
\end{split}
\end{equation*}
so that $\overline{\beta}=\beta$ by the uniqueness of $\beta$ in~\eqref{eqn:unique_beta}.

Thus we have proved:
\begin{theorem}
\label{thm:d_cofibration}
The functors $D\colon smod\text{-}SExt_M \longrightarrow \mathcal{S}mod_M$ and $d\colon cc\text{-}SExt_M\longrightarrow c\text{-}\mathcal{S}mod_M,$ where $c\text{-}\mathcal{S}mod_M\subseteq\mathcal{S}mod_M$ is the full subcategory of all $M$-semimodules $(K,\eta)$ with $K$ cancellative, are cofibrations.
\end{theorem}
(It is immediate that the same proof given for $D\colon smod\text{-}SExt_M \longrightarrow \mathcal{S}mod_M$ also works for $d\colon cc\text{-}SExt_M\longrightarrow c\text{-}\mathcal{S}mod_M.$)

The importance of this fact comes from the following classical result:
\begin{theorem}
\label{thm:cofibration_monoidal}
Let $\mathcal{U}$ be a category with finite products and ${F}\colon\mathcal{U}\longrightarrow\mathcal{V}$ a product preserving cofibration. Suppose also that the cocartesian morphisms in $\mathcal{U}$ are stable under finite products. If $(M,\omega\colon M\times M\longrightarrow M,\varepsilon\colon1\longrightarrow M)$ is an internal monoid in $\mathcal{V},$ there results a monoidal structure on the fibre ${F}^{-1}(M),$ which is symmetric as soon as the internal monoid $M$ is commutative.
\end{theorem}
(See for example~\cite[Theorem 9]{bourn-direction} for a proof.)

Observe that both functors $D\colon smod\text{-}SExt_M \longrightarrow \mathcal{S}mod_M$ and $d\colon cc\text{-}SExt_M\longrightarrow c\text{-}\mathcal{S}mod_M$ fall under the scope of Theorem~\ref{thm:cofibration_monoidal}, because, for a cofibration, being conservative is equivalent to the fact that every morphism in the domain category is cocartesian. Moreover, observe that $\mathcal{S}mod_M\cong\mathbf{CMon}(\mathcal{S}mod_M),$ because every $M$-semimodule $(K,\eta)$ is an internal monoid in $\mathcal{S}mod_M$ whose internal monoid operation in $\mathcal{S}mod_M$ is the monoid operation $+\colon K\times K\rightarrow K$ of the commutative monoid $K$ itself (the commutativity of $K$ guarantees that $+$ is a monoid homomorphism, and it is action-preserving by the action axioms on $\eta$). Similarly, we have $c\text{-}\mathcal{S}mod_M\cong\mathbf{CMon}(c\text{-}\mathcal{S}mod_M).$

By Theorem~\ref{thm:cofibration_monoidal}, we have:

\begin{corollary}
For every object $(K,\eta)$ in $\mathcal{S}mod_M$ (resp., in $c\text{-}\mathcal{S}mod_M$) there results a symmetric monoidal structure on the fibre $D^{-1}(K,\eta)$ of the direction functor $D\colon smod\text{-}SExt_M \longrightarrow \mathcal{S}mod_M$ (resp., of the functor $d\colon cc\text{-}SExt_M\longrightarrow c\text{-}\mathcal{S}mod_M$).
\end{corollary}

Since it is known from \cite[Section VI]{bourn-direction} that, in the case of group extensions of a group $G$ by a $G$-module $(A,\eta),$ the resulting tensor product on the fibres of the corresponding direction functor for groups coincides with the usual Baer sum of extensions, we are entitled to call the tensor product $\mathbb{E}\otimes\mathbb{E^\prime}$ resulting on the fibres $d^{-1}(K,\eta)$ of the functor $d\colon cc\text{-}SExt_M\longrightarrow c\text{-}\mathcal{S}mod_M$ the \emph{Baer sum} of the cc-Schreier extensions $\mathbb{E},$ $\mathbb{E^\prime}$ (and similarly for $D\colon smod\text{-}SExt_M \longrightarrow \mathcal{S}mod_M$).

Denoting by $\mathrm{SExt}(M, K, \eta)$ the set of connected components of $D^{-1}(K,\eta),$ which is nothing but the set of isomorphism classes of Schreier extensions of $M$ by $K$ inducing the action $\eta,$ we conclude that the tensor product on  $D^{-1}(K,\eta)$ endows the set $\mathrm{SExt}(M, K, \eta)$ with the structure of a commutative monoid. Thanks to \cite[Theorem 4.19]{P-II}, this commutative monoid $\mathrm{SExt}(M, K, \eta)$ is isomorphic to the second cohomology monoid $\mathcal{H}^2(M, K, \eta)$ of the cohomology theory introduced by Patchkoria in \cite{Patchkoria77}. Moreover, when $K$ is cancellative, the fibres $D^{-1}(K,\eta)$ and $d^{-1}(K,\eta)$ coincide, and it follows from \cite[Theorem 4.29]{P-II} that the corresponding commutative monoid $\mathrm{SExt}(M, K, \eta)$ is also isomorphic to the second cohomology monoid $H^2(M, K, \eta)$ of the cohomology theory introduced by Patchkoria in \cite{P-I}.

The interpretation of the higher cohomology monoids of such cohomology theories is material for future work.

\appendix
\section{Monomorphisms and regular epimorphisms in $cc\text{-}SExt_M$}
\label{appendix:monos}
We aim to prove that the monomorphisms (resp., regular epimorphisms) in the category $cc\text{-}SExt_M$ of cc-Schreier extensions on a monoid $M$ are precisely the morphisms
\begin{equation}
\begin{aligned}
\label{eqn:oto}
\xymatrix{
\E: {K} \ar@{}[rd]|{(\ast)} \ar@<3ex>[d]_-{\alpha_1} \ar@{>->}[r]^-k &X \ar[d]_-{\alpha} \ar@{->>}[r]^-f &M \ar@{=}[d]\\
\E': {K^\prime} \ar@{>->}[r]_-{k^\prime} &{X^\prime} \ar@{->>}[r]_-{f^\prime} &{M}
}
\end{aligned}
\end{equation}
in $cc\text{-}SExt_M$ such that $\alpha_1$ and $\alpha$ are monomorphisms (resp., regular epimorphisms) in $\mathbf{Mon}.$ This characterisation allows us to give the full details of the proof of Proposition~\ref{prop:d_monos_reg_epis}.

\subsection*{Existence of kernel pairs in $cc\text{-}SExt_M$}
Recall the following very general fact:
\begin{lemma}
\label{lemma:known}
Consider a commutative diagram
\[
\xymatrix{
{A} \ar[d]_-{\alpha}  \ar[r]^-l &B \ar[d]_-{\beta} \ar[r]^-g &C \ar@{=}[d]\\
{A^\prime} \ar@{>->}[r]_-{k} &{B^\prime} \ar[r]_-{f} &{C}
}
\]
in a pointed category $\mathcal{C},$ where $(A^\prime,k)$ is a kernel of $f.$ Then $(A,l)$ is a kernel of $g$ if and only if the left square is a pullback.
\end{lemma}

Consider then the kernel pairs of $\alpha_1$ and $\alpha$ in $\mathbf{Mon},$ together with the induced morphism $k\times_{X'} k\colon{\mathrm{Eq}(\alpha_1)} \to {\mathrm{Eq}(\alpha),}$ to obtain the commutative diagram
\begin{equation}
\begin{aligned}
\label{eqn:iti}
\xymatrix{
{\overline{\mathbb{E}}:}\\
{\mathbb{E}:}\\
{\mathbb{E^\prime}:}
}
\xymatrixcolsep{3pc}
\xymatrix{
{\mathrm{Eq}(\alpha_1)} \ar@{}[rd]|{(\dag)} \ar@<-.6ex>[d]_-{\pi_1} \ar@<.6ex>[d]^-{\pi_2} \ar[r]^-{k\times_{X'} k} &{\mathrm{Eq}(\alpha)} \ar@<-.6ex>[d]_-{\rho_1} \ar@<.6ex>[d]^-{\rho_2} \ar[r]^-{f\rho_1=f\rho_2} &M\ar@{=}[d] \\
{K} \ar@{}[rd]|{(\ast)} \ar[d]_-{\alpha_1} \ar@{>->}[r]^-k &X \ar[d]_-{\alpha} \ar@{->>}[r]^-f &M \ar@{=}[d]\\
{K^\prime} \ar@{>->}[r]_-{k^\prime} &{X^\prime} \ar@{->>}[r]_-{f^\prime} &{M}.
}
\end{aligned}
\end{equation}
It follows that the two commutative squares $(\dag)$ are pullbacks, since $(\ast)$ is a pullback. By Lemma~\ref{lemma:known}, we conclude that $\overline{\mathbb{E}}$ is an extension of monoids in the sense of~\eqref{eqn:ses}. (Observe that $f\rho_1(=f\rho_2)$ is a regular epimorphism, because so are $f$ and $\rho_1$ and $\mathbf{Mon}$ is a regular category.)

As $K$ is commutative and cancellative and $\mathrm{Eq}(\alpha_1)$ is (isomorphic to) a submonoid of $K\times K,$ it is clear that $\mathrm{Eq}(\alpha_1)$ is also commutative and cancellative, and we have:
\begin{proposition}
\label{prop:kernel_pairs}
The sequence $\overline{\mathbb{E}}$ is a cc-Schreier extension and~\eqref{eqn:iti} is a kernel pair of $(\alpha_1,\alpha)\colon\mathbb{E}\rightarrow\mathbb{E^\prime}$ in $cc\text{-}SExt_M.$
\end{proposition}
\begin{proof}
If $u_m$ is a representative of $m\in M$ for $\mathbb{E}$ then $(u_m,u_m)$ is a representative of $m$ for $\overline{\mathbb{E}}.$ Indeed, observe that if $(x,y)\in\mathrm{Eq}(\alpha),$ then $f(x)=f^\prime\alpha(x)=f^\prime\alpha(y)=f(y),$ so that $x=kq(x)+u_m$ and $y=kq(y)+u_m$ using~\eqref{eqn:x=kq(x)+u} (where $m=f(x)=f(y)).$ Since $\alpha(x)=\alpha(y),$ it follows that $\alpha(kq(x))+\alpha(u_m)=\alpha(kq(y))+\alpha(u_m)$; consequently, $\alpha(kq(x))=\alpha(kq(y))$ because of Lemma~\ref{lemma:rep_invertible}(1) and the fact that $\alpha$ preserves representatives. The couple $(q(x),q(y))$ is in $\mathrm{Eq}(\alpha_1),$ because $k^\prime\alpha_1(q(x))=\alpha (kq(x))=\alpha (kq(y))=k^\prime\alpha_1(q(y))$ and $k^\prime$ is monomorphic. Then $(x,y)=k\times_{X'} k(q(x),q(y))+(u_m,u_m),$ for the unique element $(q(x),q(y))\in \mathrm{Eq}(\alpha_1).$
Moreover, by Proposition~\ref{prop:rep_pres} the morphisms $\rho_1$ and $\rho_2$ preserve all representatives.

Next, we need to prove that the commutative square
\[
\xymatrixcolsep{2.5pc}
\xymatrix{
{\overline{\mathbb{E}}} \ar[d]_-{(\pi_1,\rho_1)} \ar[r]^-{(\pi_2,\rho_2)} &{\mathbb{E}} \ar[d]^-{(\alpha_1,\alpha)} \\
{\mathbb{E}} \ar[r]_-{(\alpha_1,\alpha)} &{\mathbb{E^\prime}}
}
\]
is a pullback in $cc\text{-}SExt_M.$ Suppose that two morphisms $(t_1,t),(h_1,h)$ of cc-Schreier extensions
\begin{equation*}
\begin{aligned}
\xymatrix{
{\mathbb{F}:}\\
{\mathbb{E}:}
}
\xymatrix{
{L}  \ar@<-.6ex>[d]_-{t_1} \ar@<.6ex>[d]^-{h_1} \ar@{>->}[r]^-{l} &{Y} \ar@<-.6ex>[d]_-{t} \ar@<.6ex>[d]^-{h} \ar@{->>}[r]^-{g} &M\ar@{=}[d] \\
{K} \ar@{>->}[r]_-k &X \ar@{->>}[r]_-f &M
}
\end{aligned}
\end{equation*}
are such that $\alpha_1 t_1=\alpha_1 h_1$ and $\alpha t=\alpha h.$ Then, by the universal property of the kernel pairs, we get two unique monoid homomorphisms
\begin{equation}
\begin{aligned}
\label{eqn:phi_1_phi}
\xymatrix{
{L} \ar@/^1pc/[rrd]^-{h_1} \ar@/_1pc/[rdd]_-{t_1} \ar@{.>}[rd]|{\varphi_1=\la t_1,h_1\ra} &\ &\ \\
&{\mathrm{Eq}(\alpha_1)} \pullback \ar[d]_-{\pi_1} \ar[r]^-{\pi_2} &K \ar[d]^-{\alpha_1} \\
&K \ar[r]_-{\alpha_1} &{K^\prime,}
}
\
\xymatrix{
{Y} \ar@/^1pc/[rrd]^-{h} \ar@/_1pc/[rdd]_-{t} \ar@{.>}[rd]|{\varphi=\la t,h\ra} &\ &\ \\
&{\mathrm{Eq}(\alpha)} \pullback\ar[d]_-{\rho_1} \ar[r]^-{\rho_2}&X \ar[d]^-{\alpha} \\
&X \ar[r]_-{\alpha} &{X^\prime.}
}
\end{aligned}
\end{equation}
Since $\rho_1$ and $\rho_2$ are jointly monomorphic and
\[
\begin{cases}
\rho_1 k\times_{X'} k\varphi_1=k\pi_1\varphi_1=kt_1=tl=\rho_1\varphi l\\
\rho_2 k\times_{X'} k\varphi_1=k\pi_2\varphi_1=kh_1=hl=\rho_2\varphi l,
\end{cases}
\]
it follows that $k\times_{X'} k\varphi_1=\varphi l.$ Moreover, $f\rho_1\varphi=ft=g,$ so that
\begin{equation*}
\begin{aligned}
\xymatrix{
{\mathbb{F}:}\\
{\overline{\mathbb{E}}:}
}
\xymatrix{
{L} \ar[d]_-{\varphi_1} \ar@{>->}[r]^-{l} &{Y} \ar[d]_-{\varphi} \ar@{->>}[r]^-{g} &M\ar@{=}[d] \\
{\mathrm{Eq}(\alpha_1)} \ar@{>->}[r]_-{k\times_{X'} k} &{\mathrm{Eq}(\alpha)} \ar@{->>}[r]_-{f\rho_1} &M
}
\end{aligned}
\end{equation*}
is a morphism of monoid extensions. It is a morphism of Schreier extensions, because $\varphi=\la t,h\ra$ and, by assumption, $t$ and $h$ preserve the representatives. If $(\psi_1,\psi)\colon\mathbb{F}\longrightarrow\overline{\mathbb{E}}$ is another morphism of Schreier extensions such that $\pi_1\psi_1=t_1,$ $\rho_1\psi=t$ and $\pi_2\psi_1=h_1,$ $\rho_2\psi=h,$ it immediately follows that $\psi_1=\varphi_1$ and $\psi=\varphi$ by the uniqueness of $\varphi_1$ and $\varphi$ in~\eqref{eqn:phi_1_phi}.
\end{proof}

\subsection*{Characterisation of monomorphisms and regular epimorphisms}
Now we can prove that~\eqref{eqn:oto} is a monomorphism in $cc\text{-}SExt_M$ if and only if $\alpha_1$ and $\alpha$ are monomorphisms in $\mathbf{Mon},$ and similarly for regular epimorphisms. We shall proceed by steps.

\begin{proposition}
\label{prop:reg_is_reg}
If $\alpha_1$ and $\alpha$ are regular epimorphisms in $\mathbf{Mon},$ then~\eqref{eqn:oto} is a regular epimorphism in $cc\text{-}SExt_M.$
\end{proposition}

\begin{proof}
We show that under our assumptions $(\alpha_1,\alpha)$ is a coequaliser of its kernel pair~\eqref{eqn:iti}. Suppose that $(\beta_1,\beta)\colon\mathbb{E}\longrightarrow\mathbb{F}$
\[
\xymatrix{
{\overline{\mathbb{E}}:}\\
{\mathbb{E}:}\\
{\mathbb{F}:}
}
\xymatrixcolsep{3pc}
\xymatrix{
{\mathrm{Eq}(\alpha_1)} \ar@<-.6ex>[d]_-{\pi_1} \ar@<.6ex>[d]^-{\pi_2} \ar[r]^-{k\times_{X'} k} &{\mathrm{Eq}(\alpha)} \ar@<-.6ex>[d]_-{\rho_1} \ar@<.6ex>[d]^-{\rho_2} \ar[r]^-{f\rho_1=f\rho_2} &M\ar@{=}[d] \\
{K} \ar[d]_-{\beta_1} \ar@{>->}[r]^-k &X \ar[d]_-{\beta} \ar@{->>}[r]^-f &M \ar@{=}[d]\\
{L} \ar@{>->}[r]_-{l} &{Y} \ar@{->>}[r]_-{g} &{M}
}
\]
is a morphism of cc-Schreier extensions such that $(\beta_1,\beta)(\pi_1,\rho_1)=(\beta_1,\beta)(\pi_2,\rho_2).$ Then, as $\alpha_1$ is a coequaliser of $(\pi_1,\pi_2)$ and $\alpha$ is a coequaliser of $(\rho_1,\rho_2),$ there are unique factorisations
\begin{equation*}
\begin{aligned}
\xymatrix{
{\mathrm{Eq}(\alpha_1)} \ar@<.6ex>[r]^-{\pi_1} \ar@<-.6ex>[r]_-{\pi_2} &K \ar@{->>}[r]^-{\alpha_1} \ar[d]_-{\beta_1} &{K^\prime} \ar@{.>}[ld]^-{\gamma_1}\\
&{L,}
}
\
\
\xymatrix{
{\mathrm{Eq}(\alpha)} \ar@<.6ex>[r]^-{\rho_1} \ar@<-.6ex>[r]_-{\rho_2} &X \ar@{->>}[r]^-{\alpha} \ar[d]_-{\beta} &{X^\prime} \ar@{.>}[ld]^-{\gamma}\\
&{Y.}
}
\end{aligned}
\end{equation*}
It is easy to check that $(\gamma_1,\gamma)\colon\E' \to \mathbb{F}$ is a morphism of monoid extensions such that $(\beta_1,\beta)=(\gamma_1,\gamma)(\alpha_1,\alpha).$ The uniqueness of such a morphism $(\gamma_1,\gamma)$ follows from the fact that $\alpha_1$ and $\alpha$ are (regular) epimorphisms. We are left to prove that $(\gamma_1,\gamma)\colon\mathbb{E^\prime}\longrightarrow\mathbb{F}$ is a morphism of Schreier extensions, i.e. that $\gamma$ preserves the representatives. By Proposition~\ref{prop:rep_pres}, it is enough to show that for every $m\in M$ some representative of $m$ for $\E'$ is preserved by $\gamma.$ Since $\gamma\big(\alpha(u_m)\big)=\beta(u_m)$ for every representative $u_m$ of $m$ for $\E$ and, by assumption, both $\alpha$ and $\beta$ preserve the representatives, the result follows.
\end{proof}

Concerning monomorphisms, it is clear that if $\alpha_1$ and $\alpha$ are monomorphisms in $\mathbf{Mon}$ then~\eqref{eqn:oto} is a monomorphism in $cc\text{-}SExt_M.$ The converse is also true (as we show next), and we can complete our characterisation of monomorphisms and regular epimorphisms in $cc\text{-}SExt_M$ as follows:

\begin{proposition}
\label{prop:char_complete}
Consider a morphism~\eqref{eqn:oto} of cc-Schreier extensions. Then:
\begin{enumerate}
 \item[\em{(1)}] $(\alpha_1,\alpha)$ is a monomorphism in $cc\text{-}SExt_M$ if and only if $\alpha_1$ and $\alpha$ are monomorphisms in $\mathbf{Mon};$
 \item[\em{(2)}] $(\alpha_1,\alpha)$ is a regular epirmorphism in $cc\text{-}SExt_M$ if and only if $\alpha_1$ and $\alpha$ are regular epimorphisms in $\mathbf{Mon}.$
\end{enumerate}
\end{proposition}

\begin{proof}
Consider a (regular epimorphism, monomorphism) factorisation $\alpha=ne$
\[
\xymatrix{
X \ar[rr]^\alpha \ar@{->>}[rd]_-e &\ &{X^\prime} \\
&{I_\alpha} \ar@{>->}[ru]_-n &\
}
\]
of $\alpha$ in $\mathbf{Mon},$ and let $(P,n_1,\lambda)$ be a pullback of $n$ along $k^\prime.$ Then in the diagram
\begin{equation}
\begin{aligned}
\label{eqn:fact}
\xymatrix{
{\mathbb{E}:}\\
{\mathbb{P}:}\\
{\mathbb{E^\prime}:}
}
\xymatrixcolsep{3pc}
\xymatrix{
K \ar@{}[rd]|{(a)} \ar@/_1.5pc/[dd]_(.72){\alpha_1} \ar@{>->}[r]^-k \ar[d]^-{e_1} &X \ar@{->>}[d]_-e \ar@/^1.2pc/[dd]^(.75)\alpha \ar@{->>}[r]^-f &M \ar@{=}[d] \\
P \ar@{>->}[r]^{\lambda} \pullback \ar@{>->}[d]^-{n_1} &{I_\alpha} \ar@{>->}[d]_-n \ar[r]^-{f^\prime n} &M \ar@{=}[d] \\
{K^\prime} \ar@{>->}[r]_-{k^\prime} &{X^\prime} \ar@{->>}[r]_-{f^\prime} &{M,}
}
\end{aligned}
\end{equation}
where $e_1$ is the unique morphism induced by $\alpha_1$ and $ek$ through the pullback $P,$ the square $(a)$ is also a pullback (using the fact that the square $(\ast)$ in~\eqref{eqn:oto} is a pullback, as we already argued). Thus $e_1$ is a regular epimorphism in $\mathbf{Mon},$ and $\alpha_1=n_1 e_1$ is a (regular epimorphism, monomorphism) factorisation of $\alpha_1.$ Moreover, observe that $f^\prime n$ is a regular epimorphism, because so is $(f^\prime n)e=f,$ and by Lemma~\ref{lemma:known} $(P,\lambda)$ is a kernel of $f^\prime n.$

We prove that the extension of monoids $\mathbb{P}$ is a cc-Schreier extension and that $(e_1,e),$ $(n_1,n)$ are morphisms of Schreier extensions. First observe that $P$ is commutative and cancellative (as, by the monomorphism $n_1,$ it can be realised as a submonoid of $K^\prime$). Next, we claim that, for every representative $u_m$ of $m$ for $\E,$ $e(u_m)$ works as a representative of $m$ for $\mathbb{P}.$ Indeed, given $y\in I_\alpha$ we have $y=e(x)$ for some $x\in X$ (because $e$ is a regular epimorphism in $\mathbf{Mon},$ i.e. a surjective monoid homomorphism). Using~\eqref{eqn:x=kq(x)+u}, we get $y=e(x)=e\big(kq(x)+u_m\big)=e(kq(x))+e(u_m)=\lambda(e_1q(x))+e(u_m).$ Suppose that $y=\lambda(z)+e(u_m)$ for some other $z\in P.$ We have $z=e_1(b)$ for some $b\in K$ (because $e_1$ is a regular epimorphism), thus
\begin{equation*}
\begin{split}
	e\big(kq(x)+u_m\big)=y=\lambda e_1(b)+e(u_m)=e\big(k(b)+u_m\big) \Rightarrow\\
	\alpha kq(x)+\alpha(u_m)=ne\big(kq(x)+u_m\big)=ne\big(k(b)+u_m\big)=\alpha k(b)+\alpha(u_m) \Rightarrow \\
	\alpha kq(x)=\alpha k(b) \Rightarrow  \\
	k^\prime\alpha_1(q(x))=k^\prime\alpha_1(b) \Rightarrow \\
	\alpha_1(q(x))=\alpha_1(b) \Rightarrow\\
	e_1(q(x))=e_1(b)=z,
\end{split}
\end{equation*}
using the fact that $\alpha(u_m)$ is a representative, Lemma~\ref{lemma:rep_invertible}(1) and that $k^\prime$ and $n_1$ are monomorphic. Then $\mathbb{P}$ is a Schreier extension and $e$ preserves the representatives by construction. Since $\alpha=ne$ and $\alpha$ preserves the representatives, we conclude that $n$ preserves the representatives as well (using Proposition~\ref{prop:rep_pres}).

By Proposition~\ref{prop:reg_is_reg}, $(e_1,e)\colon\mathbb{E}\longrightarrow\mathbb{P}$ is a regular epimorphism in $cc\text{-}SExt_M,$ and as $n_1$ and $n$ are monomorphisms of monoids, $(n_1,n)$ is monomorphic in $cc\text{-}SExt_M.$ From $(\alpha_1,\alpha)=(n_1,n)(e_1,e),$ it follows that if $(\alpha_1,\alpha)$ is a monomorphism in $cc\text{-}SExt_M,$ so is $(e_1,e),$ which implies that $(e_1,e)$ is an isomorphism in $cc\text{-}SExt_M.$ This, in turn, implies that $e_1$ and $e$ are isomorphisms in $\mathbf{Mon},$ and we conclude that $\alpha_1=n_1e_1$ and $\alpha=ne$ are monomorphic in $\mathbf{Mon}.$ Similarly, if $(\alpha_1,\alpha)$ is a regular epimorphism in $cc\text{-}SExt_M,$ then $n_1$ and $n$ are isomorphisms of monoids in $\mathbf{Mon},$ and we conclude that $\alpha_1=n_1 e_1$ and $\alpha=ne$ are regular epimorphisms in $\mathbf{Mon}.$
\end{proof}

\begin{remark}
By the previous proposition, every morphism~\eqref{eqn:oto} in $cc\text{-}SExt_M$ admits a factorisation~\eqref{eqn:fact} as a regular epimorphism followed by a monomorphism, and we know from Proposition~\ref{prop:kernel_pairs} that kernel pairs exist. Moreover, it is not difficult to prove that regular epimorphisms are preserved by pullbacks in $cc\text{-}SExt_M$ whenever these exist, so that $cc\text{-}SExt_M$ is a regular category in the sense of \cite[I.1.3]{barr} and \cite[Definition 2.1.1]{borceux-2}. Beware, though, that $cc\text{-}SExt_M$ is \emph{not} finitely complete, in general: indeed, the property of being closed under pullbacks fails even in the case of group extensions. The problem here is that the pullback of two morphisms $\alpha\colon f\rightarrow g,$ $\beta\colon h\rightarrow g$
\[
\xymatrix{
X \ar[rd]_-f \ar[r]^-\alpha &Y \ar[d]^-g &Z \ar[ld]^-h \ar[l]_-\beta \\
&G &\
}
\]
in $\mathbf{Gp}/G$ is constructed by means of a pullback
\[
\xymatrix{
{X\times_Y Z} \ar[r]^-{\pi_2} \ar[d]_-{\pi_1} \pullback &Z \ar[d]^-\beta \\
X \ar[r]_-\alpha &Y
}
\]
of $\beta$ along $\alpha$ in $\mathbf{Gp},$ but even if $f,$ $g$ and $h$ are surjective, the composite morphism $f\pi_1=h\pi_2$ need not be surjective, in general. The same situation happens in $cc\text{-}SExt_M.$
\end{remark}

\subsection*{The direction functor preserves and reflects monomorphisms and regular epimorphisms}
We can now prove that the functor $d\colon cc\text{-}SExt_M\longrightarrow cc\text{-}SPt_M$ preserves and reflects monomorphisms and regular epimorphisms. By the above results, a morphism~\eqref{eqn:oto} in $cc\text{-}SExt_M$ is a monomorphism if and only if both $\alpha_1$ and $\alpha$ are monomorphisms in $\mathbf{Mon},$ and similarly for regular epimorhisms. The same characterisation holds for monomorphisms and regular epimorphisms in $cc\text{-}SPt_M.$

Suppose that by applying $d$ to  a morphism~\eqref{eqn:oto} in $cc\text{-}SExt_M$ we obtain a monomorphism
\begin{equation}
\begin{aligned}
\label{eqn:da}
\xymatrix{
{K} \ar[d]_-{\alpha_1} \ar@{>->}[r]^-{\overline{\kappa}} &{df} \ar[d]_-{d(\alpha)} \ar@<.5ex>[r]^-{\overline{f}} &M \ar@{=}[d] \ar@<.5ex>[l]^-{\overline{s}} \\
{K^\prime} \ar@{>->}[r]_-{\overline{\kappa^\prime}} &{df^\prime} \ar@<.5ex>[r]^-{\overline{f^\prime}} &{M} \ar@<.5ex>[l]^-{\overline{s^\prime}}
}
\end{aligned}
\end{equation}
of cc-Schreier points. Then, in particular, $\alpha_1$ is monomorphic in $\mathbf{Mon},$ and the Short Five Lemma (Proposition~\ref{prop:short_five_lemma_schreier}) for cc-Schreier extensions guarantees that $\alpha$ is also a monomorphism of monoids; this entails that~\eqref{eqn:oto} is a monomorphism in $cc\text{-}SExt_M,$ so that $d$ reflects monomorphisms. Conversely, if~\eqref{eqn:oto} is a monomorphism in $cc\text{-}SExt_M,$ by Proposition~\ref{prop:char_complete} $\alpha_1$ is a monomorphism in $\mathbf{Mon}.$ By the Split Short Five Lemma for Schreier points (see \cite{schreier_book}, Proposition $2.3.10$), $d(\alpha)$ is also a monomorphism of monoids, and we conclude that~\eqref{eqn:da} is a monomorphism in $cc\text{-}SPt_M,$ i.e. that $d$ preserves monomorphisms. A similar argument shows that $d$ reflects and preserves regular epimorphisms, which completes the proof of Proposition~\ref{prop:d_monos_reg_epis}.

We conclude by observing that the characterisation of monomorphisms and regular epimorphisms we got for $cc\text{-}SExt_M$ does not make use of cancellativity, so the same proof works for the category $smod\text{-}SExt_M.$

\end{document}